\documentclass[11pt,tbtags,reqno]{amsart}
\usepackage{amssymb}
\usepackage{amsthm}
\usepackage[swedish, english]{babel}
\usepackage{psfrag}   
\usepackage{graphicx}
\usepackage[retainorgcmds]{IEEEtrantools}
\usepackage{bbm}
\usepackage[square, comma, numbers, sort&compress]{natbib}
\usepackage{enumerate}
\usepackage{todonotes}

\theoremstyle{plain}
\newcommand{\E}{\mathbb E}
\newcommand{\R}{\mathbb R}

\newcommand{\F}{\mathcal F}

\newcommand{\ep}{\epsilon}

\newcommand{\dir}[1]{\delta_{#1}}

\newcommand{\ud}{\,\mathrm{d}}
\newcommand{\vd}{\mathrm{d}}\newcommand{\lt}{\left}
\newcommand{\rt}{\right}
\newcommand{\pt}{\partial}

\newcommand{\tend}{\rightarrow}
\newcommand{\into}{\rightarrow}
\newcommand{\Ex}[1]{\mathbb{E}\left[ #1 \right]}
\newcommand{\gm}{\gamma}

\newcommand{\Dr}{B} % a letter to denote the unknown drift

\def\P{{\mathbb P}}

\newcommand{\Ind}{\mathbbm{1}}

\newenvironment{example}[1][Example]{\begin{trivlist}
\item[\hskip \labelsep {\bf Example}]}{\end{trivlist}}
\newenvironment{remark}[1][Remark]{\begin{trivlist} \item[\hskip \labelsep {\bf Remark}]}
{\end{trivlist}}

\newtheorem{theorem}{Theorem}[section]
\newtheorem{lemma}[theorem]{Lemma}
\newtheorem{corollary}[theorem]{Corollary}
\newtheorem{proposition}[theorem]{Proposition}

\theoremstyle{definition}
\newtheorem*{acknowledgement}{Acknowledgement}

\title{Bayesian sequential testing of the drift of a Brownian motion}
\author[Erik Ekstr\"om and Juozas Vaicenavicius]{Erik Ekstr\"om$^{1}$ and Juozas Vaicenavicius}
\subjclass[2000]{Primary 62L10, 60G40; Secondary 62C10}
\keywords{Bayesian analysis; sequential hypothesis testing; optimal stopping}
\address{Uppsala University, Box 480, 75106 Uppsala, Sweden.}
\date{\today}
\thanks{$^1$ Support from the Swedish Research Council (VR) is gratefully acknowledged.}

\begin{document}

\begin{abstract}
We study a classical Bayesian statistics problem of sequentially testing the sign of the drift of an arithmetic Brownian motion 
with the $0$-$1$ loss function and a constant cost of observation per unit of time for general prior distributions. 
The statistical problem is reformulated as an optimal stopping problem with the current conditional probability that the drift is non-negative as the underlying process. The volatility of this conditional probability process is shown to be non-increasing in time, which enables us to prove monotonicity and continuity of the optimal stopping boundaries as well as to characterize them completely in the finite-horizon case as the unique continuous solution to a pair of integral equations. In the infinite-horizon case, the boundaries are shown to solve another pair of integral equations and a convergent approximation scheme for the boundaries is provided. Also, we describe the dependence between the prior distribution and the long-term asymptotic behaviour of the boundaries.
\end{abstract}

\maketitle

\section{Introduction}

\label{introduction}
One of the classical questions in Sequential Analysis concerns the testing of two simple hypotheses about the sign of the drift of an 
arithmetic Brownian motion. More precisely, suppose that an observed process $X_t$ is an arithmetic Brownian motion
\[X_t=\Dr t +W_t,\]
where the constant $\Dr$ is unknown and $W$ is a standard driftless Brownian motion. Based
on observations of the process $X$, one wants to test sequentially the hypotheses $H_0: \Dr < 0$ and $H_1:\Dr \geq 0$. 
In the Bayesian formulation of this sequential testing problem, the drift $\Dr$ is a random variable with distribution $\mu$, 
corresponding to the hypothesis tester's prior belief about the likeliness of the different values $\Dr$ may take. Moreover, it is assumed that $\Dr$ and $W$ are independent. 
In this article, we consider a classical formulation of the problem in which the accuracy and urgency of a decision is governed by a $0$-$1$ loss function together with a constant cost $c>0$ of observation per unit time. The $0$-$1$ loss function means that the tester gains nothing for a right decision but pays a penalty of size $1$ for being wrong. The overall goal is to find a decision rule minimising the expected total cost (provided such a decision rule exists). 
If the decision is required to be made before a fixed predetermined time $T>0$, 
the problem is said to have a finite horizon, and if there is no upper bound on the decision time, an infinite horizon.  

In the classical literature \cite{C} by Chernoff and \cite{B} by Bather on Bayesian sequential testing procedures for the sign of a drift, the special case of a normal prior distribution is studied. While Bather considers the $0$-$1$ loss function described
above as well as a few other alternatives, Chernoff deals with a different penalty function, which equals the magnitude $|\Dr|$ of the error. In these papers, it is argued that the sequential analysis problem reduces to a free-boundary problem for a function of time and the current value of the observation process, but, as in the case of most time-dependent free-boundary problems, the free-boundary problem lacks an explicit solution. 
Instead, the focus of these and many \mbox{follow-up} articles in the area, including \cite{BY72}, \cite{BC}, \cite{C65a}, \cite{C65b}, and \cite{tL88} to mention a few, is on asymptotic approximations for optimal stopping boundaries (for more references, see the survey article \cite{tL97}). 
Only recently, \cite{MZ13} has characterised the optimal stopping boundaries for the original Chernoff's problem in terms of an 
integral equation, which can be solved numerically. 

In \cite{S}, the sequential testing problem is solved explicitly for a two-point prior distribution by utilising the connection 
with a time-homogeneous free-boundary problem. 
Notably, the natural spatial variable in this free-boundary problem is not the value of the observation process, but 
the conditional probability of the drift taking one of the two possible values. (Since there is a one-to-one correspondence
between these two processes at each fixed time, the free-boundary problem could be transformed into one 
based on the observation process instead, but that formulation would introduce time-dependencies and thus 
make the explicit solution more difficult to find.) 

The fact that the problem can be solved in a very special case of a \mbox{two-point} prior, raises a natural question -- 
can the sequential testing problem be solved for a more general prior distribution? 
%In this article, we investigate the sequential testing problem under a general prior distribution and arrive at a 
%fairly satisfactory even though partial positive answer. 
In this article, we investigate the sequential testing problem under a general prior distribution. Since this introduces 
time-dependencies in the problem, there is generally no hope for explicit solutions. Nevertheless, additional structure is found, 
which enables us to arrive at a fairly satisfactory answer.

To explain in some further detail, 
following standard arguments, the statistical problem is shown to admit an equivalent formulation as an optimal stopping problem, 
which we study to characterise optimal decision rules. The underlying process of the optimal stopping problem is chosen to be
the current probability, conditional on observations of $X$, that the drift is non-negative, i.e.
\[\Pi_t:=\P(B\geq 0\vert \mathcal F^X_t).\]
%
%
%For each fixed time there is an explicit one-to-one correspondence between $\Pi$ and the observation process $X$, 
%so the problem could alternatively be formulated with $X$ as the underlying process. An advantage with using 
%the $X$-process would be that
%its dynamics in terms of an innovation process are easily expressed in terms of $X$, whereas the corresponding 
%expression for the volatility of the process $\Pi$ is much more involved.
%Nevertheless, we choose to work with the $\Pi$ as the underlying, despite the cumbersome formula for its volatility, 
%mainly for two reasons.
The pay-off function of the associated optimal stopping problem is then concave in $\Pi$, 
so general results about preservation of concavity 
for optimal stopping problems may be employed to derive structural properties of the continuation region. Moreover, 
the volatility of the underlying process $\Pi$ can be shown to be decreasing in time (except for 
the two-point distribution discussed above, in which case it is constant). 
These important facts allow us to show that the optimal stopping boundaries are monotone, so techniques from the theory of 
free-boundary problems with monotone boundaries can be applied. In particular, 
the monotonicity of the boundaries enables us to prove the smooth-fit condition and the continuity of the boundaries, as well as
to study the corresponding integral equations.
In the finite-horizon case, we characterise the optimal boundaries as a unique continuous solution to a pair of integral equations.
In the infinite-horizon case, the situation turns out to be more subtle. The boundaries are shown to solve another pair of integral equations, but whether the system admits a unique solution remains unanswered. Instead, we provide a converging approximation scheme for the optimal stopping boundaries, establishing that the optimal boundaries of the finite-horizon problem converge pointwise to the optimal boundaries of the infinite-horizon problem. Also, we determine the long-term asymptotes of the boundaries and 
describe their dependence on the prior distribution.

From a technical perspective, we tackle
a number of issues stemming from the infinite-dimensionality of the parameter space of the underlying process $\Pi$, the particular form of the unbounded payoff function, as well as the presence of time-dependent infinite-horizon boundaries. 
Filtering and analytic techniques are used to understand the behaviour of the conditional probability $\Pi$, with a particular focus on the properties that are invariant under any prior distribution. Also, the generality of the prior makes the verification of the smooth-fit condition more involved than in standard situations. 
Moreover, the specific form of the payoff function with the additive unbounded time term requires some additional effort to prove optimality of the hitting time in the infinite-horizon case. 
Our approach to approximate the optimal infinite-horizon boundaries could possibly be utilised in other similar situations.  

%%As described above, our focus is to study the sequential testing problem for the sign of the drift of a Brownian motion
%%for general priors, but for a fixed `$0-1$' loss function.

Although we manage to solve this classical sequential hypothesis testing problem for the sign of the drift of a Brownian motion in the general prior case, further generalisations appear to be challenging.
The assumption that the observation process is a Brownian motion 
gives a Markovian structure to the problem in the sense that the posterior distribution then only depends on the current 
value of the observation process. In contrast, if the observation process was a more general diffusion process, 
possibly involving time and level-dependent coefficients, then the posterior distribution of the drift would depend on 
the whole observation path, and thus our one-dimensional Markovian structure would be lost. 
Another 
noticeable 
limitation is the fixed $0$-$1$ loss function, which allows us to use a convenient process $\Pi$ as the underlying diffusion. 
%! a few words added.
For more general loss functions, the natural process to express the stopping problem in terms of is the underlying observation 
process itself. However, then the corresponding boundaries are not necessarily monotone, so the corresponding free-boundary 
problem is less tractable.  For Chernoff's loss function and the particular case of a normal prior distribution, \cite{MZ13} finds an
appropriate scaling and time-change of $X$ 
%%which 
in that the transformation
guarantees the monotonicity of the corresponding stopping boundaries. The extension to more general prior distributions
and other loss functions remains an important and challenging open problem.

The paper is organised as follows. In  Section~\ref{2}, the sequential testing problem is formulated and reduced to an optimal
stopping problem. 
In Section~\ref{3}, filtering techniques are applied to find an expression for $\Pi$ in terms of the observation process $X$, and its dynamics in terms of the innovation process are determined.
We also study the volatility function of $\Pi$, and it is shown that this function is non-increasing in time.
In Section~\ref{optstop}, the optimal stopping problem is studied together with the corresponding free-boundary problem, 
and it is shown that the optimal stopping boundaries are continuous. In Section~\ref{S:inteqn}, integral equations
for the boundaries are determined, and uniqueness of solutions is established in the finite-horizon case. 
The long-term asymptotic behaviour of the infinite-horizon boundaries is presented in Section~\ref{longterm}.
Finally, Section~\ref{examples} is devoted to a special case of the normal prior distribution.

\begin{acknowledgement}
We are most grateful to Ioannis Karatzas for his suggestion to generalise the setting of an early version
of the current paper, and for sharing his unpublished notes on a related problem with us.
\end{acknowledgement}

\section{Problem formulation and reduction to an optimal stopping problem}
\label{2}

Let $(\Omega,\P,\F)$ be a complete probability space supporting a Brownian motion $W$ and a random variable $\Dr$ with distribution $\mu$ such that $W$ and $B$ are independent. Define
\[X_t=\Dr t +W_t.\] 
Writing $\mathbb F^X= \lt\{ \mathcal F^X_t \rt\}_{t \geq 0}$ for the filtration
generated by the process $X$ and the null sets in $\F$, our goal is to find a pair $(\tau, d)$ consisting of an $\mathbb F^X$-stopping time $\tau$ and 
an $\F^X_{\tau}$-measurable decision rule $d:\Omega \into \{0,1\}$, indicating which of the 
hypotheses $H_0:B<0$ or $H_1:B\geq 0$ to accept, in order to minimise the Bayes risk
\begin{IEEEeqnarray*}{rCl}
R(\tau, d) &:=& \E[\Ind_{\{d =1, \Dr<0\}} ]+ \E[\Ind_{\{d =0, \Dr\geq 0\}} ] +c\E[\tau].
\label{E:BRisk}
\end{IEEEeqnarray*}
Since $d$ is $\F^X_\tau$-measurable, we have
\begin{IEEEeqnarray}{rCl}
R(\tau, d) &=& \Ex{\Ex{\Ind_{\{ \Dr<0\}} \vert \F^X_\tau }\Ind_{\{d=1\}} + \Ex{\Ind_{\{ \Dr \geq 0\}} \vert \F^X_\tau }\Ind_{\{ d=0 \}}+c \tau} \label{E:sum},
\end{IEEEeqnarray}
which shows that, at a given stopping time $\tau$, the decision rule 
\begin{IEEEeqnarray*}{rCl}
	d = 
	\begin{cases}
	1  &\text{if } \P \lt( \Dr \geq 0\vert \F^X_\tau \rt) \geq \P \lt(  \Dr<0\vert \F^X_\tau  \rt), \\
	0  &\text{otherwise}
	\end{cases}
\end{IEEEeqnarray*}
is optimal. Consequently, writing  
\begin{equation*}
\Pi_t:=\P(\Dr\geq 0\vert \mathcal F^X_t),
\end{equation*}
the sequential testing problem (\ref{E:BRisk}) reduces to an optimal stopping problem
\begin{IEEEeqnarray}{rCl}
\label{V}
V = \inf_{\tau\in\mathcal T}\E[g(\Pi_\tau) + c\tau],
\end{IEEEeqnarray}
where $g(\pi)=\pi\wedge(1-\pi)$ and $\mathcal T$ denotes the set of $\mathbb F^X$-stopping times. 
We also consider the same sequential testing problem but with a finite horizon $T<\infty$.
The corresponding optimal stopping problem is then 
\begin{IEEEeqnarray*}{rCl}
V^T = \inf_{\tau\in\mathcal T^{T}}\E[g(\Pi_\tau) + c\tau],
\end{IEEEeqnarray*}
where $\mathcal T^{T}=\{\tau\in \mathcal T:\tau\leq T\}$.

\begin{remark}
By translation, our study readily extends to testing the hypotheses $H_0: B<\theta$ and $H_1:B\geq \theta$ for any given $\theta\in\R$.
The methods also extend to the case when the two types of possible errors are associated with different costs,
i.e. when 
\begin{IEEEeqnarray*}{rCl}
R(\tau, d) &=& a\E[\Ind_{\{d =1, \Dr<0\}} ]+ b\E[\Ind_{\{d =0, \Dr\geq 0\}} ] +c\E[\tau]
\end{IEEEeqnarray*}
for constants $a>0$ and $b>0$ with $a\not=b$. 
For simplicity of the presentation, however, 
we assume throughout the article that $\theta=0$ and $a=b=1$. 
\end{remark}

Note that in the cases when $\mu\lt((-\infty, 0)\rt)=0$ or 
$\mu\lt([0, \infty)\rt)=0$, the sequential testing problem becomes trivial as we can make the correct statement about the sign of the drift at time zero. Hence, from now onwards, we always assume that 
\begin{IEEEeqnarray}{rCl}
\label{massonbothsides}
0<\mu\lt([0, \infty)\rt)<1.
\end{IEEEeqnarray}

%The following theorem describes the optimal decision rule for the perpetual problem. 
%
%\begin{theorem}
%\label{main}
%Assume that the prior distribution $\mu$ satisfies the integrability condition \eqref{integrability} below.
%Then there exist two continuous functions $b_1,b_2:[0,\infty)\to (0,1)$ such that the decision rule $(\tau^*,d^*)$ given by
%$\tau^*:=\inf\{s\geq 0:\Pi_s\notin (b_1(s),b_2(s))\}<\infty$ a.s. and 
%\[d^*:=\left\{\begin{array}{ll}
%0 & \text{if }\Pi_{\tau^*}=b_1(\tau^*)\\
%1 & \mbox{if }\Pi_{\tau^*}=b_2(\tau^*)\end{array}\right.\]
%is optimal.
%The function $b_1$ is non-decreasing, $b_2$ is non-increasing, and the pair $(b_1,b_2)$ solves the coupled pair of 
%integral equations \eqref{integralequations-perpetual}.
%\end{theorem}
%
%Theorem~\ref{main} is proved in Sections~\ref{3}-\ref{S:inteqn} below.
%Moreover, further properties of $b_1$ and $b_2$ are exhibited in Sections~\ref{optstop}-\ref{longterm}.
%A corresponding result holds also for the finite horizon case $T<\infty$. In fact, in that case the boundaries constitute the {\em unique} solution
%of the corresponding coupled integral equations \eqref{integralequations} in the class of continuous functions (see Theorem \ref{inteqn}).
%Although annoying from a theoretical perspective, the lack of a uniqueness statement for solutions of the integral equations
%in the perpetual case plays a minor role from an applied point of view since 
%the infinite-horizon optimal stopping boundaries can be approximated
%by the optimal stopping boundaries of the finite-horizon problem (see Theorem \ref{optimal}). 

\section{Conditional probability of non-negative drift}

\label{3}

In this section, we derive a filtering equation for the distribution of $\Dr$ conditional 
on the observations of $X$, which is then applied to prove some elementary results concerning the conditional 
distribution of the sign of $\Dr$.
We also show that there is an explicit one-to-one correspondence between $\Pi$ and the observation process $X$ at each fixed time,
and we determine the dynamics of $X$ and $\Pi$ in terms of the innovation process. 

\subsection{Filtering of the unknown drift}
\begin{proposition}
\label{T:integrals}
Assume that $q:\R\to\R$ satisfies $\int_\R |q(b)| \mu(\vd b) < \infty$. Then 
\begin{IEEEeqnarray}{rCl}
\label{integrals}
\E\lt[ q(\Dr) \vert \mathcal F^X_t\rt] 
=\frac{\int_\R q(b) e^{bX_t -b^2t/2} \mu(\vd b)}{\int_\R e^{bX_t-b^2t/2} \mu(\vd b)}
\end{IEEEeqnarray} 
for any $t>0$.
\end{proposition}

\begin{proof}
The proof is based on standard methods in filtering theory, see e.g. \cite[Section 3.3]{BC09}, yet we include it for completeness.
First define an enlarged filtration $\mathbb G = \{\mathcal G_t\}_{0\leq t < \infty}$ 
as the completion of $\lt\{ \sigma (\Dr, W_s: 0 \leq s \leq t) \rt\}_{0\leq t <\infty}$. 
Clearly, $\mathcal F^X_t \subseteq \mathcal G_t$ for any $t \geq 0$, so $\mathbb G$ is an enlargement of $\mathbb F^X$. 
Observing that $Z_t := e^{-\Dr W_t-\Dr^2 t/2 }$ is a $\mathbb G$-martingale, we define 
a new probability measure $\P_*$ on the restriction $(\Omega, \mathcal G_T)$ for some large enough $T$ by
\[ \frac{\vd \P_*}{\vd \P} \vert_{\mathcal G_T} := Z_T.\] 

It can be shown that under $\P_*$, $X_t$ is a Brownian motion independent of $\mathcal G_0$ and therefore also of 
$\Dr$ and that the law of $B$ is $\mu$ (see \cite[Proposition 3.13]{BC09}). Thus Bayes' rule (cf., for example, \cite{bO}) gives that
\begin{IEEEeqnarray*}{rCl}
\E\lt[ q(\Dr) \vert \mathcal F^X_t\rt] 
= \frac{\E_*\lt[q(\Dr)/Z_t \vert \mathcal F^X_t \rt] }{\E_* \lt[ 1/Z_t \vert \mathcal F^X_t \rt]}  
=\frac{\int_\R q(b) e^{bX_t -b^2t/2} \mu(\vd b)}{\int_\R e^{bX_t-b^2t/2} \mu(\vd b)}
\end{IEEEeqnarray*} 
for $t>0$ and for any function $q:\R\to\R$ with $\int_\R |q(b)| \mu(\vd b) < \infty$. 
\end{proof}

\subsection{Conditional probability of non-negative drift}
According to Proposition~\ref{T:integrals}, 
\begin{IEEEeqnarray*}{rCl}
\Pi_t &=&  \E\lt[ \Ind_{[0, \infty)}(\Dr) \vert \mathcal F^X_t\rt] = \pi(t, X_t),
\end{IEEEeqnarray*}
where the function $\pi(t, x): (0, \infty) \times \R \into (0,1)$ is given by
\begin{IEEEeqnarray}{rCl}
\label{E:pi}
\pi(t,x) := \frac{\int_{[0,\infty)}e^{bx -b^2t/2} \mu(\vd b)}{\int_\R e^{bx-b^2t/2} \mu(\vd b)}.
\end{IEEEeqnarray}
Denoting by
\begin{IEEEeqnarray}{rCl}
\label{E:mu}
\mu_{t,x}(\vd b):=\frac{e^{bx-\frac{b^2}{2}t} \mu(\vd b)}{\int_\R e^{bx-\frac{b^2}{2}t} \mu(\vd b)}
\end{IEEEeqnarray}
the distribution of $\Dr$ at time $t$ conditional on $X_t=x$, we thus have 
\[\pi(t,x) =\int_{[0,\infty)} \mu_{t,x}(\vd b).\]

\begin{proposition}
\label{incr}
Assume that $q:\R\to\R$ is non-decreasing and satisfies $\int_\R |q(b)| \mu(\vd b) < \infty$.
Then the function
\begin{IEEEeqnarray}{rCl}
\label{u}
u(t,x):=\frac{\int_\R q(b) e^{bx -b^2t/2} \mu(\vd b)}{\int_\R e^{bx-b^2t/2} \mu(\vd b)}=\int_\R q(b)  \mu_{t,x}(\vd b)
\end{IEEEeqnarray}
is non-decreasing in $x$ for any fixed $t>0$.
\end{proposition}

\begin{proof}
We will prove the claim by showing that $u(t, \cdot)$ is differentiable with a non-negative derivative on $\R$.
By a standard differentiation lemma (see, for example, \cite[Theorem 6.28]{aK08}), both the numerator and the 
denominator in \eqref{u} are differentiable with respect to $x$ with their derivatives obtained by differentiating under the integral sign. 
Thus the derivative of $u$ with respect to the second argument $x$ is given by
\begin{IEEEeqnarray}{rCl}
\label{E:ux}
\pt_2 u(t, x) &=&
\int_{\R} q(b) b  \mu_{t,x}(\vd b) - \int_{\R} q(b) \mu_{t,x}(\vd b) \int_{\R} b \mu_{t,x}(\vd b) \\
\nonumber
&=& \E_{t,x}\lt[q(\Dr)\Dr\rt]-\E_{t,x}\lt[q(\Dr)\rt]\E_{t,x}\lt[\Dr\rt],
\end{IEEEeqnarray}
where $\E_{t,x}$ is the expectation operator under the probability measure $\P_{t,x}(\cdot ):=\P(\cdot\vert X_t=x)$.
Since 
\begin{IEEEeqnarray}{rCl}
\label{posder}
&&\hspace{-10mm}
\E_{t,x}\lt[q(\Dr)\Dr\rt]-\E_{t,x}\lt[q(\Dr)\rt]\E_{t,x}\lt[\Dr\rt]\\ 
\notag
&=& \E_{t,x}\lt[(\Dr-\E_{t,x}[\Dr])(q(\Dr)-q(\E_{t,x}[\Dr]))\right]
\geq 0,
\end{IEEEeqnarray}
this finishes the proof.
\end{proof}

\begin{corollary} \label{T:mon}
Let $a \in \R$ and $t> 0$.  Then
\begin{enumerate}
\item $\P\lt( \Dr >  a \,|\, X_t = x\rt)$ is non-decreasing in $x$, 
\item $\P\lt( \Dr < a \,|\, X_t = x\rt)$ is non-increasing in $x$. 
\end{enumerate}
\end{corollary}

\begin{proof}
The first claim follows by applying Proposition~\ref{incr} to the function $q(b)=1_{(a,\infty)}(b)$.
The second claim follows from $\P\lt( \Dr < a \,|\, X_t = x\rt)= 1-\P\lt( \Dr \geq  a \,|\, X_t = x\rt)$ and
by applying Proposition~\ref{incr} to the function $q(b)=1_{[a,\infty)}(b)$.
\end{proof}

\begin{proposition}
\label{f-1-1}
For any given $t> 0$, the function $\pi(t,\cdot): \R \to (0,1)$ defined in \eqref{E:pi} is a strictly increasing continuous bijection.
\end{proposition}

\begin{proof}
First note that 
\[(\Dr-\E_{t,x}[\Dr])(\Ind_{[0,\infty)}(\Dr)-\Ind_{[0,\infty)}(\E_{t,x}[\Dr]))\geq 0\] 
and that 
\[\P_{t,x}\left((\Dr-\E_{t,x}[\Dr])(\Ind_{[0,\infty)}(\Dr)-\Ind_{[0,\infty)}(\E_{t,x}[\Dr])) >0 \right)>0\]
since \eqref{massonbothsides} implies $\mu_{t,x}\lt((-\infty,0)\rt)>0$ and $\mu_{t,x}\lt([0, \infty)\rt)>0$.
Consequently, the inequality in \eqref{posder} is strict, so $x\mapsto \pi(t,x)$ is strictly increasing.

Next, note that
\[
	\pi(t, x) = \frac{1}{1+A(t, x)},
\]
where
\begin{IEEEeqnarray}{rCl}
\label{E:A}
	A(t,x) = \frac{\int_{(-\infty, 0)} e^{bx-  \frac{b^2}{2}t} \mu(\vd b)}{\int_{[0, \infty)} e^{bx-  \frac{b^2}{2}t} \mu (\vd b)}.
\end{IEEEeqnarray}
By monotone convergence, we find that $A(t,x) \tend 0$ as $x \tend \infty$ and $A(t,x) \tend \infty$ as $x \tend -\infty$. 
Consequently, $\lim_{x  \to\infty} \pi(t,x) =1$ and $\lim_{x \to -\infty} \pi(t,x) =0$, which finishes the proof.
\end{proof}

An immediate consequence of Proposition~\ref{f-1-1} is that for any fixed $t>0$, the spatial
inverse $\pi(t, \cdot)^{-1}: (0,1) \to \R$ exists. To facilitate intuition, we denote the inverse by $x(t, \cdot)$.

We end this subsection with a result that describes the long-term behaviour of the process $\Pi$.

\begin{proposition} 
\label{T:PiInf}
$\Pi_t \tend \Pi_\infty$ a.s. as $t \tend \infty$, where $\Pi_\infty$ is a Bernoulli random variable 
with $\P(\Pi_\infty=0)=\mu((-\infty, 0))$ and $\P(\Pi_\infty=1)=\mu([0, \infty))$.
\end{proposition}

\begin{proof} 
Firstly, since $\Pi_t=\E[\Ind_{[0,\infty)} (B) \,|\, \F_t^X]$ is a bounded martingale, by the martingale convergence theorem, the pointwise limit $\Pi_\infty := \lim_{t \tend \infty} \Pi_t$ is a well-defined random variable closing the martingale $\Pi$.
By the law of large numbers for Brownian motion and Proposition~\ref{T:levelcurves} below (the proof of which is independent of
the current result), for any $b<0$ in the support of $\mu$ we have 
\[
\P\lt( \Pi_\infty =0 \,|\, B=b\rt)=1,
\]
so
\[
\P\lt( \Pi_\infty =0\rt)= \int_{(-\infty, 0)} \P \lt(\Pi_\infty =0 \,|\, B=b\rt) \mu (\vd b) = \mu((-\infty, 0)).
\]
Hence as $\Pi_\infty$ can only take values in $[0,1]$, the fact that $\E[\Pi_\infty]=\E[\Pi_0]= \mu([0,\infty))$ implies $\P\lt( \Pi_\infty =1\rt)=\mu([0,\infty))$.  
\end{proof}

\subsection{SDE for the conditional probability of non-negative drift}
Assuming that $\Dr$ has a first moment, the conditional expectation of $\Dr$ exists and is given by 
\begin{IEEEeqnarray}{rCl}
\label{condexp}
\E[ \Dr\vert \mathcal F_t^X] = \frac{\int_\R b e^{bX_t -b^2t/2} \mu(\vd b)}{\int_\R e^{bX_t-b^2t/2} \mu(\vd b)},
\end{IEEEeqnarray}
compare \eqref{integrals}.
Moreover, the observation process $X$ is represented in terms of the innovation process
\[\hat W_t:=X_t-\int_0^t \E[ \Dr\vert \mathcal F_s^X]\,\vd s\]
as
\[\vd X_t=\E[ \Dr\vert \mathcal F_t^X]\,\vd t + \vd\hat W_t.\]
Here $\hat W$ is a standard $\mathbb F^X$-Brownian motion (see \cite[Proposition 2.30 on p.~33]{BC09}). Moreover, writing $\mathbb F^{\hat W}=\{\F^{\hat W}_t\}_{t \geq 0}$ for the completion of the filtration \mbox{$\{ \sigma(\hat W_s : 0\leq s \leq t) \} _{t \geq 0 }$}, we have $\mathbb F^X=\mathbb F^{\hat W}$ (see the remark on p.~35 in \cite{BC09}). 

From now onwards, the following integrability condition on $\mu$ will be imposed throughout the article.
\begin{flalign}
\label{integrability}
\text{\bf{Assumption. }} &  \int_\R e^{\ep b^2}\mu(\vd b)<\infty \quad \text{for some } \ep>0.&
\end{flalign}
Note that this assumption is a minor restriction on our hypothesis testing problem since, given any probability distribution $\mu$, 
the distributions $\mu_{t,x}$ all satisfy \eqref{integrability} for $t>0$. In other words, no matter what prior distribution $\mu$ 
one starts with, the condition \eqref{integrability} will be satisfied after any infinitesimal time of observation. Also, note that 
the assumption allows us to extend the definition of $\mu_{t,x}$ in \eqref{E:mu} to $t=0$.
Moreover, if we have a prior distribution $\xi$ on $\R$ given by
\[	
	\xi(\vd b) := \frac{e^{\ep b^2}\mu(\vd b)}{\int_\R e^{\ep b^2}\mu(\vd b)},
\]
then
\begin{IEEEeqnarray}{rCl} \label{E:mualt}
\mu_{0,x} (\vd b) = \xi_{2\ep, x}(\vd b) := \frac{e^{bx-b^2 (2\ep)/2} \xi(\vd b)}{\int_\R e^{bx-b^2 (2\ep)/2} \xi(\vd b) }.
\end{IEEEeqnarray}
Consequently, the distribution $\mu_{0,x}$ can be identified with a conditional distribution at time 0 given that the prior distribution 
at time $-2\ep$ was $\xi$ and the current value of the observation process is $x$.
This gives us a generalisation of the notion of the starting point of the observation process $X$ to allow \mbox{$X_0=x \neq 0$},
and we may regard time 0 as an interior point of the time interval.

A closer look at the condition \eqref{integrability} and the expression \eqref{E:pi} assures that the standard differentiability lemma can be applied to differentiate $\pi(t,x)$ with respect to both variables multiple times inside $(-2\ep,\infty)\times \R$. Applying Ito's formula to $\Pi_t=\pi(t,X_t)$, we find that 
\begin{IEEEeqnarray*}{rCl}
\vd \Pi_t &=& \left(\pt_1\pi(t,X_t) +\E[ \Dr\vert \mathcal F_t^X]\pt_2\pi(t,X_t) + \frac{1}{2}\pt^2_2\pi(t,X_t)\right)\ud t+
\pt_2 \pi(t,X_t) \ud \hat W_t  \\
&=& \pt_2 \pi(t,x(t,\Pi_t)) \ud \hat W_t, 
\end{IEEEeqnarray*}
where the second equality is verified using the expression \eqref{condexp} and
\begin{equation*}
	\pt_1\pi(t,x) = - \int_{[0,\infty)} \frac{b^2}{2} \mu_{t,x}(\vd b) + \int_{[0,\infty)} \mu_{t,x}(\vd b) \int_\R \frac{b^2}{2} \mu_{t,x}(\vd b),
\end{equation*}
\begin{equation} \label{E:d2Pi}
\pt_2\pi(t,x) =  \int_{[0,\infty)} b \mu_{t,x}(\vd b) -  \int_\R b \mu_{t,x}(\vd b) \int_{[0,\infty)} \mu_{t,x}(\vd b), 
\end{equation}
\begin{IEEEeqnarray*}{rCl} 
\pt^2_2\pi(t,x) &=& \int_{[0,\infty)} b^2 \mu_{t,x}(\vd b) -2 \int_\R b \mu_{t,x}(\vd b) \int_{[0,\infty)} b \mu_{t,x}(\vd b) \\
&& -\int_\R b^2  \mu_{t,x}(\vd b) \int_{[0, \infty)} \mu_{t,x}(\vd b)  + 2 \int_{[0, \infty)} \mu_{t,x}(\vd b) \lt( \int_\R b \mu_{t,x}(\vd b) \rt)^2.
\end{IEEEeqnarray*}
Thus the dynamics of $\Pi_t$ are specified by a zero drift and the volatility 
\begin{IEEEeqnarray}{rCl}
\label{E:sig}
\sigma(t,\Pi_t)=\pt_2 \pi(t,x(t, \Pi_t)) ,
\end{IEEEeqnarray}
being a positive function of the current time and the current value of $\Pi$.
Using \eqref{E:ux}, the volatility function can be expressed as
\begin{IEEEeqnarray}{rCl}
\label{sigma}
\notag
\sigma(t,\pi) &=&
\E_{t,x(t,\pi)}[\Dr \Ind_{\{ \Dr\geq0\}}] - \P_{t,x(t,\pi)}\lt( \Dr \geq 0  \rt) \E_{t,x(t,\pi)}[\Dr ]\\
&=& (1-\pi)\E_{t,x(t,\pi)}[\Dr \Ind_{\{ \Dr\geq 0\}}] -\pi \E_{t,x(t,\pi)}[\Dr \Ind_{\{ \Dr< 0\}}].
\end{IEEEeqnarray}

\begin{example}{\bf (The two-point distribution).}
Assume that $\P(\Dr=a_1)=1-p$ and $\P(\Dr=a_2)=p$ for some constants $a_1<0 \leq a_2$ and $p\in (0,1)$.
Then
\[\P(B=a_2\vert X_t=x)=\pi(t,x)=\frac{p e^{a_2x-a_2^2t/2}}{(1-p)e^{a_1x-a_1^2t/2} +pe^{a_2x-a_2^2t/2}},\]
and
\[\sigma(t,\pi)=(a_2-a_1)\pi(1-\pi).\]
This example with a two-point prior distribution is a special case of the Wonham filter.
\end{example}

\begin{example}{\bf (The normal distribution).}
Assume that $\mu$ is the normal distribution with mean $m$ and variance $\gamma^2$. Then
the conditional distribution $\P(\cdot\vert X_t=x)=\mu_{t,x}$ is also normal but with mean $\frac{m +\gamma^2x}{1+t\gamma^2}$
and variance $\frac{\gamma^2}{1+t\gamma^2}$.
Consequently, 
\begin{IEEEeqnarray}{rCl} \label{E:pinormal}
\pi(t,x)=\Phi\left(\frac{m+\gamma^2 x}{\gamma\sqrt{1+t\gamma^2}}\right)
\end{IEEEeqnarray}
and
\[\sigma(t,\pi)=\varphi(\Phi^{-1}(\pi))\frac{\gamma}{\sqrt{1+ t\gamma^2}},\]
where
\[\varphi(z)=\frac{1}{\sqrt{2\pi}} e^{-z^2/2}\,dz\]
and
\[\Phi(y)=\int_{-\infty}^y \varphi(z)\,dz\]
are the density and the cumulative distribution of the standard normal random variable, respectively.
Note that this instance of a normal prior distribution is a special case of the well-known Kalman-Bucy filter, 
see for example \cite[Chapter 6]{bO}.
\end{example}

\subsection{Volatility of the conditional probability process}
\label{4}

In this section we study the volatility function $\sigma$. The main result, Corollary~\ref{T:SD}, states that the 
volatility is non-increasing as a function of time.

Let $\pi \in (0,1)$ be a fixed number and consider the map $x(\cdot, \pi) : [0,\infty) \into \R$ sending $t \mapsto x(t, \pi)$. Note that the graph of this function is the trajectory that the process $X_t$ has to follow in order for the conditional probability process $\Pi_t$ to stay constant at the value $\pi$. Thus we call $x(\cdot, \pi)$ \emph{the $\pi$-level curve}. 
Some handy regularity of $x(\cdot, \cdot)$ and $\sigma(\cdot, \cdot)$ is brought to light in the following.

\begin{proposition} 
\label{T:C1}
The functions
	$x(\cdot, \cdot) : [0, \infty) \times (0,1) \into \R$ and
	$\sigma(\cdot, \cdot) : [0, \infty) \times (0,1) \into \R$ are both $C^1$. 
%	Moreover, for any $t \geq 0$, the volatility function $\sigma(t, \pi) \tend 0$ both as $\pi \tend 0$ and as $\pi \tend 1$. 
\end{proposition}

\begin{proof}
Define $F: (-\ep, \infty) \times \R \into (-\ep, \infty) \times (0,1)$ by $F(t, x) = (t, \pi(t, x))$, where $\ep$ is as in the assumption (\ref{integrability}). The function $F$ is $C^1$, which is evident by applying the standard differentiation lemma as in the proof of Proposition \ref{incr}. The Jacobian matrix of $F$ is
\[
	J_F(t,x)= \lt( \begin{array}{cc} 1 & 0 \\ \pt_1 \pi(t,x) & \pt_2 \pi(t,x) \end{array} \rt).
\]
Since $F$ is invertible and $\det(J_F(t,x)) = \pt_2\pi(t,x) >0$ for all $t > -\ep$ and all $x \in \R$, the inverse function theorem 
tells us that the inverse of $F$ is also $C^1$, with
\[
J_{F^{-1}}(F(t,x))= \lt(J_F(t,x) \rt)^{-1}.
\]
Consequently, $x(\cdot, \cdot)$ is $C^1$ on $(-\ep, \infty) \times (0,1)$  with the derivatives $\pt_1 x(t, \pi)=-\pt_1 \pi(t,x(t,\pi)) / \pt_2 \pi(t,x(t,\pi))$ and $\pt_2 x(t,\pi) = 1/\pt_2 \pi(t, x(t,\pi))$. 
Finally, since a product of continuous functions is continuous, by the chain rule, $\sigma(\cdot, \cdot)$ is continuously differentiable on $(-\ep, \infty) \times (0,1)$ and so on $[0, \infty) \times (0,1)$. 
\end{proof}

Next, denoting the initial value $\Pi_0$ by $\pi_0 \in (0,1)$, we show that the tails of the conditional distribution 
$\mu_{t,x}$ are decreasing along the level curve $x(\cdot,\pi_0)$.

\begin{proposition} 
\label{T:decrtails} \item
\begin{enumerate}
\item
If $a \geq 0$, then for any $t>0$,
\begin{IEEEeqnarray}{rCl}
\label{E:tails}
	\P \lt(\Dr>a | X_t = x(t, \pi_0) \rt) - \P(\Dr>a) \leq 0. 
\end{IEEEeqnarray}
Supposing $\mu((a, \infty)) >0$, the inequality above is strict if and only if $\mu ([0,a])>0$.
\item
Likewise, if $a < 0$, then for any $t>0$, 
\begin{IEEEeqnarray}{rCl}
\label{T:decrtailsn}
	\P \lt(\Dr<a | X_t = x(t, \pi_0) \rt) - \P(\Dr<a) \leq 0. 
\end{IEEEeqnarray}

Supposing $\mu((-\infty, a)) >0$, the inequality above is strict if and only if $\mu([a, 0))>0$.
\end{enumerate}
\end{proposition}

\begin{proof} 
We prove only the first of the two claims as the proof of the second one follows the same argument with straightforward modifications. 

In the case $\mu((a,\infty))=0$, the claim holds trivially with equality in (\ref{E:tails}). Thus we assume that $\mu((a,\infty))>0$ in what follows. Writing $x(t)$ instead of $x(t, \pi_0)$ for brevity, we note that using $\mu_{t, x(t)}([0, \infty)) = \mu([0, \infty))$, the inequality (\ref{E:tails}) is easily seen to be equivalent to 
\begin{IEEEeqnarray}{rCl} 
\label{E:etails1}
\frac{\int_{(a,\infty)} e^{bx(t)-b^2  \frac{t}{2}} \mu(\vd b)}{\int_{[0,\infty)} e^{bx(t)-b^2 \frac{t}{2}} \mu(\vd b)} \leq
\frac{\mu((a,\infty)) }{\mu([0, \infty))}\,.
\end{IEEEeqnarray}
Now, we will split the proof into consideration of two separate cases.

\noindent \underline{Case 1:} $a\geq 2x(t)/t$. Here 
\begin{IEEEeqnarray*}{rCl}
\frac{\int_{[0,a]}{e^{bx(t)-b^2  \frac{t}{2}}} \mu(\vd b)}{\int_{(a,\infty)}{e^{bx(t)-b^2  \frac{t}{2}}} \mu(\vd b)} &\geq& \frac{e^{ax(t)-a^2  \frac{t}{2}}\mu([0, a])}{e^{ax(t)-a^2  \frac{t}{2}}\mu((a,\infty))} \\
&=& \frac{\mu([0, a])}{\mu((a,\infty))},
\end{IEEEeqnarray*}
which is equivalent to (\ref{E:etails1}).
Since $\mu((a, \infty))>0$, the inequality above is strict if and only if \mbox{$\mu([0,a])>0$}.

\noindent \underline{Case 2:} $x(t)>0$ and $0< a < 2 x(t)/t$. Using that $\mu_{t, x(t)}([0, \infty)) = \mu([0, \infty))$, we get
\begin{IEEEeqnarray*}{rCl}
\int_{[0, \infty)} e^{bx(t) -b^2 t/2} \mu (\vd b) &=& \mu([0,\infty)) \frac{\int_{(-\infty, 0)} e^{bx(t) -b^2 t/2} \mu (\vd b)}{\mu((-\infty,0))} \nonumber \\
&<& \mu([0, \infty)),
\end{IEEEeqnarray*}
where the inequality holds since $x(t)>0$. Hence rewriting
\begin{IEEEeqnarray*}{rCl}
\frac{\int_{(a, \infty)} e^{bx(t) -b^2 t/2} \mu (\vd b)}{\mu((a,\infty))} &=& 
\frac{\int_{[0, \infty)}e^{bx(t) -b^2 t/2} \mu (\vd b) - \int_{[0,a]} e^{bx(t) -b^2 t/2} \mu (\vd b)}{\mu([0,\infty)) - \mu([0, a])}
\end{IEEEeqnarray*} 
and keeping in mind that $0< a < 2 x(t)/t$, one clearly sees that
\begin{IEEEeqnarray}{rCl}
\frac{\int_{(a, \infty)} e^{bx(t) -b^2 t/2} \mu (\vd b)}{\mu((a,\infty))}&\leq& \frac{\int_{[0, \infty)} e^{bx(t)-b^2 t/2} \mu (\vd b)}{\mu([0, \infty))} \label{E:etails3}
\end{IEEEeqnarray}
with the strict inequality if and only if $\mu([0,a])>0$. 
As (\ref{E:etails3}) is equivalent to (\ref{E:etails1}), the proof is complete. 
\end{proof}

\begin{corollary} \label{T:SD}
For any $\pi \in (0,1)$ fixed, the volatility function \mbox{$\sigma(\cdot,\pi) : [0,\infty)\tend \R$} defined in (\ref{E:sig}) is non-increasing in time. Moreover, it is strictly decreasing for any initial prior $\mu$ except a two-point distribution in which case $t \mapsto \sigma(t,\pi)$ is a constant function.
\end{corollary}

\begin{proof}
A key to the proof is a realisation that it is sufficient to prove that $\sigma(0, \pi_0) \geq \sigma(s, \pi_0)$ for any $s > 0$; the rest will immediately follow by a `moving-frame' argument. More precisely, by `moving-frame' we mean that for any $\pi \in (0,1),$ $t\geq0$, one can think of $\mu_{t, x(t,\pi)}$ as the initial prior distribution at time zero and so immediately obtain that $\sigma(t, \pi) \geq \sigma(t+s, \pi)$ for any $s>0$.

Using a shorthand $x(t)$ for $x(t, \pi_0)$ as before, recall from (\ref{sigma}) that 
\begin{IEEEeqnarray*}{rCl}
\sigma(t, \pi_0)
= (1-\pi_0)\E_{t,x(t)}[\Dr \Ind_{\{ \Dr\geq 0\}}] -\pi_0 \E_{t,x(t)}[\Dr \Ind_{\{ \Dr< 0\}}].
\end{IEEEeqnarray*} 
Consequently, 
\begin{IEEEeqnarray*}{rCl}
\sigma(0,\pi_0)-\sigma(t,\pi_0)  &=& (1-\pi_0) \left(\E[\Dr \Ind_{\{ \Dr\geq 0\}}] - \E_{t,x(t)}[\Dr \Ind_{\{ \Dr\geq 0\}}]\right) \\
     && +\pi_0 \left(\E_{t,x(t)}[\Dr \Ind_{\{ \Dr< 0\}}]-\E [\Dr \Ind_{\{ \Dr< 0\}}]\right)\\
&\geq& 0
\end{IEEEeqnarray*}
by Proposition \ref{T:decrtails}. Moreover, by the same proposition, the inequality reduces to an equality if and only if $\mu$ is a two-point distribution.
\end{proof}

\begin{remark}
It seems difficult to find an easy intuitive argument for the monotonicity of the volatility function. As an example, 
consider a symmetric prior distribution, and a strictly positive time-point $t$ at which the observation process
satisfies $X_t=0$. Then the conditional distribution $\mu_{t,0}$ is also symmetric, so $\Pi_t=\Pi_0=1/2$. 
One certainly expects that $\mu_{t,0}$ is obtained from the prior distribution $\mu$ by pushing mass towards 
zero (this is also verified in Proposition~\ref{T:decrtails} above). One could expect that the $\Pi$-process of a 
distribution with a lot of mass close to zero is sensitive to small changes in the observation process since the 
mass easily may `spill over' to the other side of zero, and thus such a distribution
gives rise to a comparatively large volatility. On the other hand, a concentrated distribution makes it 
difficult to distinguish possible drifts from each other, and changes in the observation process would to a higher degree be attributed
to the Brownian fluctuations. This implies a slow learning process, which indicates a small volatility. Corollary~\ref{T:SD}
shows that the latter effect outweighs the former one.
\end{remark}

\section{Analysis of the optimal stopping problem}
\label{optstop}

In this section, we study the perpetual optimal stopping problem \eqref{V} and its \mbox{finite-horizon} counterpart under the integrability condition \eqref{integrability}. Most of the time the emphasis is on the perpetual case, though the corresponding results also hold for the finite horizon case by the same arguments. If the analogy is straightforward, we only comment on it, otherwise, more details are provided.

\subsection{The value function with arbitrary starting points}
Recall that
\[\vd\Pi_t =\sigma(t,\Pi_t) \ud \hat W_t,\] 
where
\[\sigma(t,\pi)=\pt_2 \pi(t,x(t,\pi)) >0\] 
for all $(t, \pi) \in [0, \infty)\times (0,1)$ (beware that $\pi(\cdot, \cdot)$ is a function, while $\pi$ is a real number). 
%! I would like to say that we also study the finite horizon opt. stopping problem.
We embed the optimal stopping problem \eqref{V}, in which the starting point of the process $\Pi$ is given by
$\Pi_0:=\P(\Dr\geq 0)$, into the optimal stopping problem
\begin{IEEEeqnarray}{rCl}
\label{v}
v(t,\pi)=\inf_{\tau\in\mathcal T}\E\left[g(\Pi^{t,\pi}_{t+\tau}) +c \tau\right], \quad (t, \pi) \in [0,\infty)\times (0,1),
\end{IEEEeqnarray}
for the process $\Pi^{t,\pi}=\Pi$ given by
\begin{equation}
\left\{\begin{array}{ll} \label{E:PiSDE}
\vd \Pi_{t+s}=\sigma(t+s,\Pi_{t+s}) \ud \hat W_{t+s} \,, & (s>0)\\
\Pi_t=\pi,\end{array}\right.
\end{equation}
where $\mathcal T$ denotes the set of stopping times with respect to the completed filtration of $ \{\Pi^{t,\pi}_{t+s}\}_{s \geq 0}$. 
The SDE (\ref{E:PiSDE}) possesses a unique solution since $\sigma(\cdot, \cdot)$ is locally Lipschitz by Proposition \ref{T:C1}. 
Furthermore, the embedding has a consistent interpretation also at time $t=0$, which is given by \eqref{E:mualt} and the remark following it. Note that choosing $\tau=0$ gives $v(t,\pi)\leq g(\pi)$.

\begin{proposition}
\label{concave}
The value function $v(t,\pi)$ is concave in $\pi$ for any fixed $t\geq 0$.
\end{proposition}

\begin{proof}
This follows by a standard approximation argument using optimal stopping problems where stopping is only allowed 
at a discrete set of time-points, compare \cite{E}.

To outline this, denote by $\mathcal T_{t,n}$, where $t\leq n$, the set of stopping times in $\mathcal T$ 
taking values in $\{k 2^{-n},k=0,1,...,n2^{n}\}\cap[0,n-t]$, $n=1,2,...$, and let
\[v_n(t,\pi)=\inf_{\tau\in\mathcal T_{t,n}}\E\left[g(\Pi^{t,\pi}_{t+\tau}) + c\tau\right].\]
Then $v_n(n,\pi)=g(\pi)$ is concave in $\pi$. By preservation of concavity for martingale diffusions, 
see \cite{JT} (the results of \cite{JT} extend to the current setting with both an upper and a lower bound on the state space), 
$\pi\mapsto v_n(t,\pi)$ is concave also for $t\in (n-2^{-n},n)$. Next, at time $t=n-2^{-n}$
the value is given by dynamic programming as
\[v_n(t,\pi)=\min\left\{ g(\pi),\E \left[v_n(n,\Pi^{t,\pi}_n) + c2^{-n}\right]\right\},\]
which is concave (being the minimum of two concave functions). Proceeding recursively shows that 
$v_n$ is concave in $\pi$ at all times $t\in[0,n]$. Since $v_n$ converges pointwise to $v$ as $n\to\infty$, this implies
that also $v$ is concave in $\pi$.
\end{proof}

\begin{proposition}
\label{v-is-nondecr}
The value function $v(t,\pi)$ is non-decreasing in $t$ for every fixed $\pi\in(0,1)$. 
\end{proposition}

\begin{proof}
This can be proven using approximation by Bermudan options as in the proof of Proposition~\ref{concave} above.
Indeed, for a fixed time $t\geq 0$ one may approximate $v(t,\pi)$ by the optimal value in the case when stopping times are restricted to
take values in the set $\{k 2^{-n}: n \in \mathbb{N}, \, k \in \{0,1,\ldots,n2^{n}\}\}$. Since the expected value of a concave function of a
martingale diffusion is non-increasing in the volatility, see \cite{JT}, the approximation is non-decreasing in $t$ by Corollary~\ref{T:SD}. Letting $n\to\infty$ finishes the proof.
\end{proof}

\begin{remark}
It is straightforward to check that the monotonicity of $v$ in the time-variable also holds in cases when the rate $c$ of the observation cost is increasing in time (instead of a constant as in our set-up). As the non-decreasing value function implies the monotonicity of the optimal stopping boundaries (see Proposition~\ref{T:b1b2}), the same monotonicity of the optimal stopping boundaries would be present also in the case of the observation rate $c$ being increasing in time. Accordingly, we expect the subsequent results to extend to that case as well.
\end{remark}

\begin{proposition}
\label{v-is-cont}
The value function $v$ is continuous on $[0,\infty)\times[0,1]$.
\end{proposition}

\begin{proof}
By concavity of $v$ in the second variable together with the bounds $0\leq v\leq g$, we have that $v$ is Lipschitz continuous
in $\pi$ for any fixed $t$, with Lipschitz coefficient 1. 
Thus it suffices to check that $v$ is continuous in time. 
To do this, 
let $t_2>t_1 \geq 0$ and note that
\[v(t_1,\pi) \geq \E[v(t_2,\Pi_{t_2}^{t_1,\pi})] \geq v(t_2,\pi)-\E[\vert \Pi_{t_2}^{t_1,\pi}-\pi\vert],\]
where the first inequality holds since $\E[v(t_2,\Pi_{t_2}^{t_1,\pi})]$ represents the value of a sequential testing problem, started at $t_1$, with 
the running cost of observation not started until time $t_2$, the second inequality holds by the concavity of $v$ in the second variable and the bounds $0\leq v\leq g$.
Thus
\[0\leq v(t_2,\pi)-v(t_1,\pi)\leq \E[\vert \Pi_{t_2}^{t_1,\pi}-\pi\vert].\]
Since the expected value of a convex function of a martingale diffusion is non-decreasing in the volatility (again by \cite{JT}) and $\sigma(0,\cdot) \geq \sigma(\cdot, \cdot)$ on \mbox{$[0,\infty)\times (0,1)$}, we deduce that
$ \E[\vert \Pi_{t_2}^{t_1,\pi}-\pi\vert] \leq \E[\vert \Pi_{t_2-t_1}^{0,\pi}-\pi\vert] \tend 0 $
as $t_2- t_1 \searrow 0$.
This finishes the proof.
\end{proof}

\begin{lemma}
\label{v-at-1/2}
We have $v(t,1/2)<g(1/2)$ for all times $t\geq 0$.
\end{lemma}

\begin{proof}
Let $t  \geq 0$, and define $A_\ep := [t, t+\ep]\times[1/2-(c+1)\ep, 1/2 +(c+1)\ep]$ for $\ep$ small enough so that $A_\ep\subseteq [t,\infty)\times(0,1)$.
Let
\[\tau_\ep:=\inf\{s\geq 0:(t+s,\Pi^{t, 1/2}_{t+s})\notin A_\ep\}\]
be the first exit time from $A_\ep$. By Proposition \ref{T:C1} and Corollary~\ref{T:SD}, $\sigma(\cdot, \cdot)$ is 
continuous and strictly positive on $[0, \infty)\times(0,1)$. Thus $\sigma_\ep := \inf_{(s,\pi) \in A_\ep} \sigma(s, \pi)$ is strictly 
positive and non-increasing as a function of $\ep$, so $\sigma_\ep$ is bounded away from 0 as $\ep\to 0$. Now, 
\begin{IEEEeqnarray*}{rCl}
g(1/2)-v(t,1/2) &\geq&
1/2- \E\lt[ g(\Pi^{t, 1/2}_{t+\tau_\ep} ) + c\tau_\ep \rt] \\
&\geq& 1/2-
(1/2  - (c+1) \ep + c \ep ) \P(\tau_\ep< \ep) - (1/2 + c\ep) \P(\tau_\ep= \ep) \\
&=& \ep -(c+1)\ep \P(\tau_\ep = \ep) \IEEEyesnumber. \label{E:ppp}
\end{IEEEeqnarray*}
Here
\begin{IEEEeqnarray*}{rCl}
\P(\tau_{\ep}=\ep)
&=&\P \lt( \sup_{0\leq s \leq\ep} \lt| \int_0^s \sigma(t+u, \Pi^{t, 1/2}_{t+u} ) \ud \hat{W}_{t+u} \rt| \leq (c+1) \ep \rt) \\
&\leq& \P\lt( \sup_{0\leq s \leq \ep} \lt| \sigma_\ep \hat{W}_s \rt| \leq  (c+1)\ep \rt) \\
&\leq& \P\lt( \sup_{0\leq s \leq \ep} \hat{W}_s \leq  (c+1)\ep / \sigma_\ep \rt) \to 0\IEEEyesnumber \label{E:tz}
\end{IEEEeqnarray*}
as $\ep \tend 0$, where the first inequality follows from \cite[Lemma 10]{JT}.
Consequently, (\ref{E:ppp}) and (\ref{E:tz}) yield that $ g(1/2)-v(t,1/2)> 0$, which finishes the proof of the claim.
\end{proof}

\subsection{The structure of an optimal strategy}
Recalling that $0\leq v(t,\pi)\leq g(\pi)$, we denote by 
\[\mathcal C :=\{(t,\pi)\in[0,\infty)\times(0,1):v(t,\pi)<g(\pi)\}\]
the continuation region, and by 
\[\mathcal D:=\{(t,\pi)\in[0,\infty)\times(0,1):v(t,\pi)=g(\pi)\}\]
the stopping region. Since $v$ is continuous, $\mathcal C$ is open and $\mathcal D$ is closed.
Resorting to intuition from optimal stopping theory, we expect that the stopping time
\begin{equation}
\label{tau}
\tau^*:=\inf\{s\geq 0:(t+s,\Pi^{t,\pi}_{t+s})\in \mathcal D\}
\end{equation}
is an optimal stopping time in \eqref{v}. 
(Note that standard optimal stopping theory does not apply since the pay-off process is not uniformly integrable.)
The optimality of $\tau^*$ is verified below, see Theorem~\ref{optimal}.

\begin{proposition} \label{T:b1b2}
There exist two functions $b_1:[0,\infty)\to[0,1/2)$ and $b_2:[0,\infty)\to(1/2,1]$ such that 
\[\mathcal C=\{(t,\pi):b_1(t)<\pi<b_2(t)\}.\]
The function $b_1$ is non-decreasing and right-continuous with left limits. Similarly, $b_2$ is non-increasing
and right-continuous with left limits.
\end{proposition}

\begin{proof}
The existence of $b_1$ and $b_2$ follows from the concavity of $v$ and Lemma~\ref{v-at-1/2}. 
The monotonicity properties are immediate consequences of Proposition~\ref{v-is-nondecr}.
Moreover, by the continuity of $v$, the function $b_1$ is upper semi-continuous and $b_2$ is lower semi-continuous. 
Hence, they are right-continuous with left limits. 
\end{proof}

Let us also consider the same optimal stopping problem with a finite horizon $T>0$. It is written as 
\begin{IEEEeqnarray}{rCl}
\label{vT}
v^T(t,\pi)=\inf_{\tau\in\mathcal T_{T-t}}\E\left[g(\Pi^{t,\pi}_{t+\tau}) + c\tau\right],
\end{IEEEeqnarray}
where $\mathcal T_{T-t}$ denotes the set of stopping times less or equal to $T-t$ with respect to the completed filtration of $\{ \Pi^{t,\pi}_{t+s}\}_{s \geq 0}$. Note that all results for the perpetual problem \eqref{v}
described above in this section also hold for the finite horizon problem \eqref{vT}, with the obvious modifications regarding the time horizon,
by the same proofs.
Moreover, the pay-off process in \eqref{vT} is continuous and bounded, so 
standard optimal stopping theory 
(see, for example, \cite[Corollary 2.9 on p.~46]{PS06}) 
yields that \[\tau^T:=\inf\{s\geq 0:\Pi^{t,\pi}_{t+s}\notin (b^T_1(t+s),b^T_2(t+s))\}\]
is an optimal stopping time in \eqref{vT}, where $b_1^T$ and $b_2^T$ are the corresponding boundaries
enclosing the finite-horizon continuation region 
\[\mathcal C^T:=\{(t,\pi)\in[0,T)\times(0,1):v^T(t,\pi)<g(\pi)\}.\] 

The infinite-horizon problem can be approximated by finite-horizon problems in the following sense.

\begin{theorem} 
\label{optimal}
The functions $v^T \searrow v$, $b_1^T \searrow b_1$, and   $b_2^T \nearrow b_2$ pointwise as $T \nearrow \infty$. The stopping times  $\tau^T \nearrow \tau^*$ a.s.~as $T\nearrow \infty$, where $\tau^*$ is defined in \eqref{tau}. Moreover,  $\tau^*$ is optimal in \eqref{v}.
\end{theorem}

\begin{proof}
Since $v^T\geq v$, we have that $b_1\leq b_1^T<b_2^T\leq b_2$ and $\tau^T\leq \tau^*$.
By bounded and monotone convergence, $v^T(t,\pi)\searrow v(t,\pi)$ pointwise as $T\to\infty$, so 
$b_1^T \searrow b_1$ and   $b_2^T \nearrow b_2$ pointwise as $T\to\infty$. Consequently, by the monotonicity 
of $b_i$ and $b_i^T$, it follows that $\tau^T\nearrow \tau^*$ a.s. 
Thus
\[v^T(t,\pi)=\E\left[ g(\Pi^{t,\pi}_{t+\tau^T})+c\tau^T\right]\to \E\left[ g(\Pi^{t,\pi}_{t+\tau^*})+c\tau^*\right]\]
by bounded and monotone convergence. By uniqueness of limits, 
\[v(t,\pi)= \E\left[ g(\Pi^{t,\pi}_{t+\tau^*})+c\tau^*\right],\]
so $\tau^*$ is optimal.
\end{proof}

\subsection{Optimal stopping boundaries and the free-boundary problem}

\begin{proposition}
The boundaries $b_1$ and $b_2$ satisfy $0<b_1(t)<1/2<b_2(t)<1$ for all times $t\geq 0$.
\end{proposition}

\begin{proof}
The two middle inequalities are granted by Lemma \ref{v-at-1/2}.
To see that $b_1>0$ on $[0, \infty)$, without loss of generality, it is sufficient to show that $b_1>0$ on $(0, \infty)$; 
this is due to the possibility provided by $(\ref{E:mualt})$ to start the process $\Pi$ slightly earlier. 
Let us assume, to reach a contradiction, that $b_1(t)=0$ for some $t>0$.
Then, by monotonicity, $b_1\equiv 0$ on $[0,t]$. By the martingale inequality,
\begin{IEEEeqnarray*}{rCl}
\P(\tau^*\leq t) \leq \P\left(\sup_{0\leq s\leq t}\Pi_s^{0,\pi}\geq 1/2\right) \leq 2\pi.
\end{IEEEeqnarray*}
Consequently, 
\[\E [\tau^*]\geq t(1-2\pi),\]
and so $0 = \lim_{\pi \searrow 0} g(\pi) \geq \lim_{\pi \searrow 0}v(0, \pi) \geq c\E[\tau^*] \geq ct > 0$, which is a clear contradiction.
%so if $\pi\leq ct/(2ct+1)$, then $v(0,\pi)\geq c\E[\tau^*]\geq \pi$. Thus $b_1(0)\geq ct/(2ct+1)>0$, which contradicts our assumption. 
Therefore $b_1>0$ at all times. The proof that $b_2<1$ is analogous.
\end{proof}

\begin{proposition} \label{T:SmoothFit}
The triplet $(v,b_1,b_2)$ satisfies the free boundary problem
\begin{IEEEeqnarray}{rCl}
\label{fbp}
\left\{\begin{array}{ll}
\partial_1 v(t,\pi)+\frac{\sigma(t,\pi)^2}{2}\partial_2^2 v(t,\pi)+c=0 & b_1(t)<\pi<b_2(t)\\
v(t,\pi)=\pi & \pi\leq b_1(t)\\
v(t,\pi)=1-\pi & \pi\geq b_2(t).\end{array}\right.
\end{IEEEeqnarray}
Moreover, the smooth-fit condition holds in the sense that the 
function $\pi\mapsto v(t,\pi)$ is $C^1$ for all $t\geq 0$.
\end{proposition}

\begin{proof}
The proof that the differential equation in \eqref{fbp} holds is based on the strong Markov property and the continuity
of $v$. However, the procedure is standard and we therefore omit the argument, 
referring to the proof of \cite[Theorem~7.7]{KS2} for the details instead.
The value of $v$ for $\pi\notin (b_1(t),b_2(t))$ follows from concavity and the definition of $b_1$ and $b_2$.

For the \mbox{smooth-fit} condition, note that the value function $\pi\mapsto v(t,\pi)$ is continuous on $(0,1)$ and 
$C^1$ for $\pi\in (b_1(t),b_2(t))$ as well as for $\pi\in (0,b_1(t))\cup(b_2(t),1)$. Thus it remains to check the $C^1$ property at $b_1(t)$ and $b_2(t)$. To prove the 
$C^1$ property at $b_1(t)$ (the $C^1$ property at $b_2(t)$ being completely analogous),
note that since $v$ is concave in $\pi$, it suffices to show that
\begin{IEEEeqnarray}{rCl}
\label{liminf}
\liminf_{\ep\downarrow 0}\frac{v(t,b(t) + \ep)-v(t,b(t))}{\ep}\geq 1.
\end{IEEEeqnarray}
Without loss of generality, we do this for $t=0$, letting $\pi=b_1(0)$. 

Let $\ep \in (0,1/2-\pi)$ and denote by $\tau^\ep$ the first hitting time of the stopping region for $\Pi^{0,\pi+\ep}$. Then
\begin{IEEEeqnarray*}{rCl}
v(0,\pi + \ep)-v(0,\pi) &\geq& \E\left[ g(\Pi_{\tau^\ep}^{0,\pi+\ep})-g(\Pi_{\tau^\ep}^{0,\pi})\right]\\
&\geq& \ep -2\E\left[(\Pi_{\tau^\ep}^{0,\pi+\ep}-\Pi_{\tau^\ep}^{0,\pi})\Ind_{\{\Pi_{\tau^\ep}^{0,\pi+\ep}>1/2\}}\right].
\end{IEEEeqnarray*}
Thus, to prove \eqref{liminf} it suffices (by the Cauchy-Schwartz inequality) to show that 
\begin{IEEEeqnarray}{rCl}
\label{CS}
\E\left[(\Pi_{\tau^\ep}^{0,\pi+\ep}-\Pi_{\tau^\ep}^{0,\pi})^2\right]\P\left(\Pi_{\tau^\ep}^{0,\pi+\ep}>1/2\right)=o(\ep^2)
\end{IEEEeqnarray}
as $\ep\to 0$.
To do this, first assume that $\sigma$ is Lipschitz continuous in $\pi$ on any compact time interval,
and define
\begin{IEEEeqnarray*}{rCl}
h(t):=\E\left[(\Pi_{t\wedge\tau^\ep}^{0,\pi+\ep}-\Pi_{t\wedge\tau^\ep}^{0,\pi})^2\right].
\end{IEEEeqnarray*}
Fixing $T>0$, for $t\in[0,T]$ we have
\begin{IEEEeqnarray*}{rCl}
h(t)&=& \E\left[\left(\ep + \int_0^{t\wedge\tau^\ep} \sigma(s,\Pi_s^{0,\pi+\ep})-\sigma(s,\Pi_s^{0,\pi})\, \vd \hat{W}_s\right)^2\right]\\
&\leq& \ep^2 + \int_0^t \E\left[D(T)^2\left( \Pi_s^{0,\pi+\ep}-\Pi_s^{0,\pi}\right)^2 \Ind_{\{s\leq \tau^\ep\}}\right] \vd s\\
&\leq& \ep^2 +D(T)^2 \int_0^t  h(s)\,\vd s,
\end{IEEEeqnarray*}
where $D(T)$ is a Lipschitz constant for $\sigma$ on $[0,T]\times(0,1)$. 
Consequently, Gronwall's inequality yields 
\begin{IEEEeqnarray}{rCl}
\label{P0}
h(T)\leq \ep^2e^{D(T)^2T}.
\end{IEEEeqnarray}
%Next, by concavity and non-negativity of $v(0,\cdot)$ we have 
%\[g(\pi+\ep)-D\ep \leq v(0,\pi+\ep)\leq g(\pi+\ep)\]
%for some constant $D$, so $c \E[\tau_\ep]\leq D\ep$. 

Next, denote by $f(y):=\pi-\frac{\pi}{1-\pi}(y-\pi)$ the affine function satisfying $f(\pi)=\pi$ and $f(1)=0$, and note that
$f\leq g$ on $[\pi,1]$. Therefore, 
\begin{IEEEeqnarray*}{rCl}
\label{E:VFineq}
c\E[\tau^\ep] &=& v(0,\pi+\ep)- \E[g(\Pi_{\tau^\ep}^{0,\pi+\ep})] \\
& \leq& g(\pi+\ep) - \E[f(\Pi_{\tau^\ep}^{0,\pi+\ep})]\\
&=& \pi+\ep-(\pi-\frac{\pi}{1-\pi}\E[\Pi_{\tau^\ep}^{0,\pi+\ep}-\pi])\\
&=& \ep/(1-\pi),
\end{IEEEeqnarray*}
where the inequality follows from  the monotonicity of $b_1$ and the last equality by optional sampling.
Thus, writing $D= 1/(1-\pi)$, we have
%
%
%Next, note that 
%\begin{IEEEeqnarray}{rCl} \label{E:RedPayoff}
%\E[g(\Pi_{\tau^\ep}^{0,\pi+\ep})] \geq	\E \big[ \Pi_{\tau^\ep}^{0,\pi+\ep} \Ind_{\{0 <\Pi_{\tau^\ep}^{0,\pi+\ep} <\frac12 \}} \big] \geq \frac{\pi (\frac{1}{2}-\pi -\ep)}{\frac{1}{2}-\pi},
%\end{IEEEeqnarray}
%where the second inequality follows from an application of the optional sampling theorem and the fact that $b_1$ is increasing. Also, recall that 
%\begin{IEEEeqnarray}{rCl} \label{E:VFineq}
% \E[g(\Pi_{\tau^\ep}^{0,\pi+\ep})] +c\E[\tau^\ep]=v(0,\pi+\ep) \leq g(\pi+\ep) =\pi +\ep \,.
%\end{IEEEeqnarray}
%Hence, writing $D:=(\frac{1}{2}-\pi -\ep)/(\frac{1}{2}-\pi)$, from \eqref{E:RedPayoff} and \eqref{E:VFineq} we obtain that $c \E[\tau_\ep]\leq D\ep$ and so
\begin{IEEEeqnarray}{rCl}
\label{P1}
\P(\tau_\ep>T)\leq D\ep/(c T).
\end{IEEEeqnarray}
Moreover, writing 
\[\tau_{\pi,1/2}:=\inf\{s\geq 0:\Pi^{0,\pi+\ep}_s\notin (\pi,1/2)\},\]
we have 
\begin{IEEEeqnarray}{rCl}
\label{P2}
\P\left(\Pi_{\tau^\ep}^{0,\pi+\ep}>1/2\right) \leq \P\left(\Pi_{\tau_{\pi,1/2}}^{0,\pi+\ep}=1/2\right) \leq \frac{\ep}{\frac{1}{2}-\pi}= C\ep
\end{IEEEeqnarray}
for $C=2/(1-2\pi)$,
where we used the martingality of $\Pi$ to obtain the second inequality.
Putting together \eqref{P0}, \eqref{P1} and \eqref{P2} yields
\begin{IEEEeqnarray*}{rCl}
&& \E\left[(\Pi_{\tau^\ep}^{0,\pi+\ep}-\Pi_{\tau^\ep}^{0,\pi})^2\right] \P\left(\Pi_{\tau^\ep}^{0,\pi+\ep}>1/2\right) \\
&\leq& \left( \E\left[(\Pi_{\tau^\ep}^{0,\pi+\ep}-\Pi_{\tau^\ep}^{0,\pi})^2 \Ind_{\{\tau_\ep\leq T\}}\right]+ \P(\tau_\ep > T)\right)
\P\left(\Pi_{\tau^\ep}^{0,\pi+\ep}>1/2\right)\\
&\leq& \ep^2 C (\ep e^{D(T)^2 T} + D/(cT)).
\end{IEEEeqnarray*}
Given $\delta>0$, it is possible to choose $T$ large enough so that $CD/(cT)\leq \delta/2$, 
and then to choose $\ep>0$ small enough so that $C \ep e^{D^2(T) T}\leq \delta/2$. This proves \eqref{CS}
and thus finishes the proof of the smooth-fit property if $\sigma$ is Lipschitz in $\pi$, locally uniformly in $t$.

For a general $\sigma$, due to the $C^1$ regularity of $\sigma$ on $[0, \infty)\times(0,1)$, one can find another volatility function $\hat\sigma$ that is Lipschitz continuous in $\pi$ on any given compact 
interval in time, and that satisfies 
$0\leq \hat\sigma\leq \sigma$ everywhere and $\hat\sigma=\sigma$ on $[0,\infty)\times[b_1(0),b_2(0)]$.
By monotonicity in the volatility, the corresponding value function $\hat v$ satisfies $\hat v\geq v$. On the other hand, 
since $\hat\sigma=\sigma$ on $[0,\infty)\times[b_1(0),b_2(0)]$ and since $\tau^*$ is optimal for the volatility $\sigma$, 
we also have $\hat v\leq v$, so $\hat v=v$. By the above argument, $\hat v$ is $C^1$, which finishes the proof.
\end{proof}

\begin{theorem}
The boundaries $b_1$ and $b_2$ are both continuous.
\end{theorem}

\begin{proof}
Let us prove continuity of $b_1$ (the proof for $b_2$ is analogous). We know that $b_1$ is right-continuous, so it suffices to assume for a contradiction that $b_1$ is not continuous at some time $t_0>0$. By monotonicity,  $b_1(t_0)>b_1(t_0-)$. 
In the continuation region, $\partial_1 v\geq 0$, so \eqref{fbp} yields
\[\frac{\sigma^2}{2}\partial^2_2 v\leq -c.\]
Since $\sigma$ is locally bounded away from zero, this means that on each compact set we can find some constant $d>0$ such that $\partial^2_2 v\leq -d$. 
By Proposition~\ref{T:SmoothFit}, the map $\pi \mapsto v(t,\pi)$ is $C^1$ on $[b_1(t), b_2(t)]$ for any $t\geq 0$, so
for $t<t_0$ and $b_1(t)<\pi<b_1(t_0)$,  we have
\begin{IEEEeqnarray*}{rCl}
v(t,\pi)-g(\pi) &=& \int_{b_1(t)}^\pi \int_{b_1(t)}^w \partial_2^2(v-g)(t,u)\ud u \ud w\\
&\leq& -d(\pi-b_1(t))^2/2.
\end{IEEEeqnarray*}
Choosing $\pi=  \frac{b_1(t_0-)+b_1(t_0)}{2}$ and letting $t\to t_0$ gives
\[v(t_0,\frac{b_1(t_0-)+b_1(t_0)}{2})-g(\frac{b_1(t_0-)+b_1(t_0)}{2})\leq -d(b_1(t_0)-b_1(t_0-))^2/2<0.\]
This contradicts the assumption that $(t_0,\frac{b_1(t_0-)+b_1(t_0)}{2})$ belongs to the stopping region, so
$b_1$ has to be continuous.
\end{proof}

\begin{remark}
Even though, in this section, all the results are formulated for the perpetual problem \eqref{v}, it is straightforward to check that the corresponding results for the finite-horizon problem \eqref{vT} also hold. In that case, the boundaries $b_1^{T}:[0,1] \into (0,1)$ and $b_2^{T}:[0,1] \into (0,1)$ are continuous and monotone, with $0< b_1^T < 1/2 < b_2^T < 1$ on $[0, T)$ and $b_1^{T}(T)=b_2^{T}(T)=1/2$. Also, the assertions of Proposition \ref{T:SmoothFit} hold for $(v^T, b_1^T, b_2^T)$ on the time interval $[0, T)$ in place of $(v, b_1, b_2)$.
\end{remark}

\section{Integral equations for the boundaries}
\label{S:inteqn}

It is well-known that optimal stopping boundaries, under some conditions, can be characterized by certain
integral equations, compare \cite{J} and \cite{P2}.
In this section, we study the integral equations for the optimal stopping boundaries arising in our sequential testing problem.
For the problem \eqref{vT} with finite horizon, a pair of integral equations is shown to completely characterise the optimal stopping
boundaries within the class of continuous solutions. The situation in the perpetual case is more delicate, and uniqueness of 
solutions remains an open question.

\subsection{A pair of integral equations for the finite-horizon boundaries}

\begin{theorem}
\label{inteqn}
Assume that $T<\infty$. Then the pair $(b^T_1,b^T_2)$ is the unique continuous solution of 
\begin{IEEEeqnarray}{rCl}
\label{integralequations}
\left\{\begin{array}{ll}
c_1(t) = \E \left[g(\Pi^{t,c_1(t)}_T)\right]+ c\int_0^{T-t} \P(c_1(t+u)<\Pi^{t,c_1(t)}_{t+u}<c_2(t+u))\,\vd u\\
1-c_2(t) =\E \left[g(\Pi^{t,c_2(t)}_T)\right]+  c\int_0^{T-t} \P(c_1(t+u)<\Pi^{t,c_2(t)}_{t+u}<c_2(t+u))\,\vd u\end{array}\right.
\end{IEEEeqnarray}
such that 
$0<c_1(t) \leq 1/2 \leq c_2(t)<1$ for all $t\in[0,T]$.
\end{theorem}

\begin{proof}
For $(t,\pi)\in [0,T)\times (0,1)$ and $\ep>0$ small enough, let
\[\tau_\ep:=\inf\{s\geq 0:\Pi_{t+s}^{t,\pi}\notin(\ep,1-\ep)\}\wedge (T-t-\ep).\]
Applying Ito's formula (more precisely, an extension of Ito's formula, 
see \cite[Theorem 3.1 and Remark 3.2]{P}, which can be applied thanks to the monotonicity of $b_1^T$ and $b_2^T$) to the  
process $v^T(t+s\wedge\tau_\ep,\Pi^{t,\pi}_{t+s\wedge\tau_\ep})$ and then taking expectations yields
\begin{IEEEeqnarray*}{rCl}
\E \left[v(t+\tau_\ep,\Pi^{t,\pi}_{t+\tau_\ep})\right]=v^T(t,\pi)-c\E\left[\int_0^{\tau_\ep}\Ind_{(b^T_1(t+u),b^T_2(t+u))}(\Pi^{t,\pi}_{t+u})\,\vd u\right].
\end{IEEEeqnarray*}
Since $\tau_\ep\to T-t$ as $\ep\to 0$, it follows from Proposition~\ref{v-is-cont} and bounded convergence that
\begin{IEEEeqnarray}{rCl}
\label{vT}
\E \left[g(\Pi^{t,\pi}_T)\right]=v^T(t,\pi)-c\int_0^{T-t}\P(b^T_1(t+u)<\Pi^{t,\pi}_{t+u}<b^T_2(t+u))\,\vd u.
\end{IEEEeqnarray}
Plugging in $\pi=b^T_1(t)$ and $\pi=b^T_2(t)$ shows that $(b^T_1,b^T_2)$ solves \eqref{integralequations}.

For uniqueness, assume that $(c_1,c_2)$ is another continuous solution to \eqref{integralequations} with
$0<c_1(t) \leq 1/2 \leq c_2(t)<1$, and define
\begin{IEEEeqnarray}{rCl}
\label{V}
V(t,\pi):=\E \left[g(\Pi^{t,\pi}_T)\right] + c\int_0^{T-t}\P(c_1(t+u)<\Pi_{t+u}^{t,\pi}<c_2(t+u))\,\vd u.
\end{IEEEeqnarray}
Then $V(t,c_1(t))= c_1(t)$ and $V(t,c_2(t))= 1-c_2(t)$ by \eqref{integralequations}, and $V(T,\pi)=g(\pi)$.
Moreover, by the Markov property, the process 
\[M_s:=
V(t+s,\Pi_{t+s}^{t,\pi})
+c\int_0^{s} \Ind_{\lt(c_1(t+u),c_2(t+u)\rt)}(\Pi^{t,\pi}_{t+u})\,\vd u\]
is a martingale for any $(t,\pi)$. In particular, the process 
\[\tilde M_s:=
v^T(t+s,\Pi_{t+s}^{t,\pi})
+c\int_0^{s} \Ind_{\lt( b^T_1(t+u), b^T_2(t+u)\rt)}(\Pi^{t,\pi}_{t+u})\,\vd u\]
is also a martingale.

\noindent
\underline{Claim 1}: $V(t,\pi)=g(\pi)$ for $\pi\notin(c_1(t),c_2(t))$.

\noindent
Assume that $\pi \leq c_1(t)$ (the case $\pi \geq c_2(t)$ is similar), and let 
\[\gamma_c:=\inf\{s\geq 0:\Pi^{t,\pi}_{t+s} \geq c_1(t+s)\}\wedge (T-t).\] 
Then
\begin{IEEEeqnarray*}{rCl}
V(t,\pi) &=& \E \left[V(t+\gamma_c,\Pi^{t,\pi}_{t+\gamma_c}) \right]
= \E \left[ \Pi^{t,\pi}_{t+\gamma_c}\right] = \pi=g(\pi),
\end{IEEEeqnarray*}
whith the first equality being justified by optional sampling and the martingale property of $M$, the second by \eqref{integralequations}, and the third by optional sampling and the martingale property of $\Pi$. 

\noindent
\underline{Claim 2}: $V\geq v^T$.

\noindent
Take $(t,\pi)$ such that $c_1(t)< \pi< c_2(t)$, and let 
\[\tau_c:=\inf\{s\geq 0:\Pi^{t,\pi}_{t+s}\notin(c_1(t+s),c_2(t+s))\}\wedge (T-t).\] 
Then
\begin{IEEEeqnarray*}{rCl}
V(t,\pi) &=& \E \left[V(t+\tau_c,\Pi^{t,\pi}_{t+\tau_c}) \right] + c\E\left[\int_0^{\tau_c} \Ind_{\lt(c_1(t+u),c_2(t+u)\rt)}(\Pi^{t,\pi}_{t+u})\,\vd u\right]\\
&=& \E \left[g(\Pi^{t,\pi}_{t+\tau_c})\right] + c\E\left[\tau_c\right] \geq v^T(t,\pi).
\end{IEEEeqnarray*}
From this and Claim 1, Claim 2 follows.

\noindent
\underline{Claim 3}: $b_1^T\leq c_1$ and $c_2\leq b^T_2$.

\noindent
Assume that $b_1^T(t)> c_1(t)$ for some $t$. Choose $\pi=c_1(t)$, and let
\[\gamma_b:=\inf\{s\geq 0:\Pi^{t,\pi}_{t+s}\geq b_1^T(t+s)\}\wedge (T-t).\]
Then, by right-continuity of $b_1^T$ and $c_1$, 
\begin{IEEEeqnarray}{rCl}
\label{cond}
\E\left[ \int_0^{\gamma_b} \Ind_{\lt(c_1(t+u),c_2(t+u)\rt)}(\Pi^{t,\pi}_{t+u})\,\vd u \right]>0.
\end{IEEEeqnarray}
On the other hand, by optional sampling and martingality of $M$ and $\tilde M$ we have
\begin{IEEEeqnarray*}{rCl}
0 &=& V(t,\pi)-v^T(t,\pi) \\
&=&
\E\left[V(t+\gamma_b,\Pi^{t,\pi}_{t+\gamma_b})-v^T(t+\gamma_b,\Pi^{t,\pi}_{t+\gamma_b})\right] \\
&&+ c\E\left[ \int_0^{\gamma_b} \Ind_{\lt(c_1(t+u),c_2(t+u)\rt)}(\Pi^{t,\pi}_{t+u})\,\vd u \right].
\end{IEEEeqnarray*}
Since $V\geq v$ by Claim 2, this contradicts \eqref{cond} and thus $b_1^T\leq c_1$.
The claim $c_2\leq b_2^T$ is proved similarly.

Now, it follows from \eqref{vT}, \eqref{V} and Claim 3 that $V=v^T$. Therefore, 
since $v^T<g$ for $\pi\in (b_1^T(t),b^T_2(t))$ it follows from Claims 1 and 3 that $b^T_1=c_1$ and $b^T_2=c_2$,
which finishes the proof.
\end{proof}

\begin{remark}
A closer inspection of the proof above shows that uniqueness holds in the larger class of pairs $(c_1,c_2)$ such that 
$c_1$ is right-continuous with no negative jumps and $c_2$ is right-continuous with no positive jumps.
\end{remark}

\subsection{A pair of integral equations for the infinite-horizon boundaries}
\begin{theorem}
\label{inteqn-perpetual}
The pair $(b_1,b_2)$ is a solution of 
\begin{IEEEeqnarray}{rCl}
\label{integralequations-perpetual}
\left\{\begin{array}{ll}
b_1(t) =  c\int_0^\infty \P(b_1(t+u)<\Pi^{t,b_1(t)}_{t+u}<b_2(t+u))\,\vd u\\
1-b_2(t) =  c\int_0^\infty \P(b_1(t+u)<\Pi^{t,b_2(t)}_{t+u}<b_2(t+u))\,\vd u.\end{array}\right.
\end{IEEEeqnarray}
\end{theorem}

\begin{proof}
For fixed $T>0$, an application of It\^{o}'s formula as in the preceding proof gives
\[\E \left[v(T,\Pi^{t,\pi}_T)\right]=v(t,\pi)-c\int_0^{T-t}\P(b_1(t+u)<\Pi^{t,\pi}_{t+u}<b_2(t+u))\,\vd u.\]
Since $v$ is bounded and $\Pi^{t,\pi}_T$ converges to either 0 or 1 as $T\to\infty$ by Proposition~\ref{T:PiInf}, we find that 
\[v(t,\pi)=c\int_0^\infty\P(b_1(t+u)<\Pi^{t,\pi}_{t+u}<b_2(t+u))\,\vd u.\]
Plugging in $\pi=b_1(t)$ and $\pi=b_2(t)$ shows that $(b_1,b_2)$ solves \eqref{integralequations-perpetual}.
\end{proof}

\begin{remark}
The main technical difficulty when trying to apply the uniqueness proof of Theorem \ref{inteqn} to the perpetual problem 
lies in the lack of a straightforward extension
of the optional sampling theorem to unbounded, possibly infinite stopping times.
\end{remark}

\subsection{The case of a symmetric volatility function}

Now assume that the volatility function is symmetric about $\pi=1/2$, i.e. 
$\sigma(t,\pi)=\sigma(t,1-\pi)$. This is the case, for example, if the prior distribution $\mu$ is symmetric about zero in the sense that 
$\mu([0,a))=\mu((-a,0))$ for all $a>0$. Then, by symmetry, $b^T_1=1-b^T_2$, and we set
$b^T:=b^T_1$. The following result is a straightforward consequence of Theorem~\ref{inteqn}.

\begin{theorem}
Assume that $\sigma$ is symmetric about $\pi=1/2$. Then
the boundary $b^T$ is the unique continuous solution of 
\begin{IEEEeqnarray}{rCl}
\label{symmetriceqn}
c(t) =\E \left[g(\Pi^{t,c(t)}_T)\right]+ c\int_0^{T-t} \P(c(t+u)<\Pi^{t,c(t)}_{t+u}<1-c(t+u))\,\vd u
\end{IEEEeqnarray}
such that $0<c(t)\leq 1/2$ for all $t\in[0,T]$.
\end{theorem}

\begin{remark}
Although not necessarily symmetric, all normal prior distributions as well as all two-point priors give rise to symmetric volatilities, compare Section~\ref{3}.
\end{remark}

\section{Long-term asymptotics of the volatility and the boundaries}

\label{longterm}

Since the boundaries $b_1$ and $b_2$ are monotone, the limits $b_i(\infty):=\lim_{t\to\infty}b_i(t)$, $i=1,2$, exist with
$b_1(\infty)\leq 1/2$ and $b_2(t)\geq 1/2$. In this section we determine these limits.
To do that, we first derive a few limiting properties of level curves
as well as study the limit $\sigma(\infty,\pi):=\lim_{t\to\infty}\sigma(t,\pi)$ of the volatility.

Let us define
\begin{IEEEeqnarray}{rCl}
\label{D:r}
r=\inf\{ s \geq 0 : \mu\lt([s, s+\ep)\rt) > 0 \text{ for all } \ep > 0\}
\end{IEEEeqnarray}
and 
\begin{IEEEeqnarray}{rCl}
\label{D:l}
l=\sup\{ s<0 : \mu \lt( (s-\ep, s] \rt)>0 \text{ for all } \ep >0 \}.
\end{IEEEeqnarray}
We write $m = (l+r)/2$ for the midpoint between $l$ and $r$.

The following proposition will serve as a useful device for understanding long-term volatility.
\begin{proposition} \label{T:levelcurves}\item
	\begin{enumerate} [1.]
		\item If $\alpha > m$, then $\pi(t,\alpha t)\to 1$ as $t\to\infty$.
		\item If $\alpha < m$, then $\pi(t,\alpha t)\to 0$ as $t \tend \infty$.
	\end{enumerate}
\end{proposition} 

\begin{proof}
Given $\alpha \in \R$, define
\[ h(t) := \frac{\int_{(-\infty,0)} \exp \lt(-(b -\alpha)^2 \frac{t}{2}\rt) \mu(\vd b)} 
					{\int_{[0,\infty)} \exp \lt( -(b -\alpha)^2 \frac{t}{2} \rt) \mu(\vd b)},\]
so that $\pi(t,\alpha t)=1/(1+h(t))$. 
We will prove the claims in two different cases separately.
\begin{enumerate}[(i)]
	\item \underline{First case: $l < r$.} 
	\begin{enumerate} [1.]
		\item First note that, in view of Proposition~\ref{f-1-1}, it suffices to treat the case $\alpha\in(m,r)$.
		For such $\alpha$, fix $\gm>r$ such that $\gamma-\alpha<\alpha-l$. 
		Then 
		\begin{IEEEeqnarray*}{rCl}
		h(t) &\leq& \frac{\exp \lt( -(\alpha-l)^2 \frac{t}{2} \rt) \int_{(-\infty,0)} \mu(\vd b)} 
{\exp \lt( -(\gamma-\alpha)^2 \frac{t}{2} \rt) \int_{[0, \gm]} \mu(\vd b)} 
			\to 0 
		\end{IEEEeqnarray*}
as $t\to\infty$.
Hence $\pi(t,\alpha t)\to 1$ as $t\to\infty$.
	\item For the second result, suppose that $\alpha < m$, and note that it suffices to treat the case $\alpha \in(l,m)$.
Let $\gamma<l$ be such that $\alpha-\gamma<r-\alpha$.
Then
	\begin{IEEEeqnarray*}{rCl}
		h(t) 	&\geq& \frac{\exp \lt( -(\alpha-\gamma)^2 \frac{t}{2} \rt) \int_{(\gamma,0)} \mu(\vd b)} 
{\exp \lt( -(r-\alpha)^2 \frac{t}{2} \rt) \int_{[0,\infty)} \mu(\vd b)} \to\infty 
		\end{IEEEeqnarray*}
as $t\to\infty$.
Hence $\pi(t,\alpha t)\to 0$ as $t\to\infty$.
	\end{enumerate}
	\item \underline{Second case: $l = r=0$.}
\begin{enumerate}[1.]
	\item Assume that $\alpha>0$, and let $\ep>0$. Then
\begin{IEEEeqnarray*}{rCl}
		h(t) &\leq& \frac{\int_{(-\infty, -\ep)} \exp \lt(-(\alpha+\ep)^2 \frac{t}{2}\rt) \mu(\vd b)+\int_{[-\ep,0)} \exp \lt(-\alpha^2 \frac{t}{2}\rt) \mu(\vd b)}{\int_{[0,\alpha]} \exp \lt(-\alpha^2 \frac{t}{2}\rt) \mu(\vd b)} \\
			&\tend& \frac{\mu([-\ep,0))}{\mu([0,\alpha])}
\end{IEEEeqnarray*}
as $t\to\infty$.
Thus, since $\ep>0$ is arbitrary and $\mu([-\ep, 0)) \tend 0$ as $\ep \tend 0$, we conclude that $h(t) \tend 0$ as $t \tend \infty$.
Consequently, $\pi(t, \alpha t)\to 1$.

	\item Next, assume that $\alpha<0$. Choosing $\gm \in (\alpha, 0)$ with $\mu \lt( (\alpha,\gamma)\rt) >0$, we find that
\begin{IEEEeqnarray*}{rCl}
		h(t) 	&\geq& \frac{\exp \lt( -(\alpha-\gamma)^2 \frac{t}{2} \rt) \int_{(\alpha,\gamma)} \mu(\vd b)} 
{\exp \lt( -\alpha^2 \frac{t}{2} \rt) \int_{[0,\infty)} \mu(\vd b)} \to\infty 
		\end{IEEEeqnarray*}
as $t\to\infty$. Consequently, $\pi(t,\alpha t)\to 0$ as $t\to\infty$, which finishes the proof.
\end{enumerate}
\end{enumerate}
\end{proof}

\begin{remark} 
Notice that Proposition~\ref{T:levelcurves} implies that for any fixed value $\pi$, the corresponding level curve $x(\cdot, \pi)$ satisfies $\lim_{t \to\infty} (ct-x(t, \pi))=\infty$ if $\alpha>m$, and $\lim_{t \to\infty} (\alpha t-x(t, \pi))=-\infty$ if $\alpha<m$.
\end{remark}

\subsection{Long-term behaviour of the volatility}

Now, we are in a position to determine the limit $\sigma(\infty,\pi):=\lim_{t\to\infty}\sigma(t,\pi)$ of the volatility as time increases.

\begin{proposition} \label{T:ltvol}
The long-term limit of volatility satisfies $\sigma(\infty,\pi)=(r-l)\pi(1-\pi)$.
\end{proposition}

\begin{remark}
Note that if $l=r=0$, then the volatility converges to zero as time tends to infinity. Also, note that if $l<r$, then the volatility tends to the volatility from the case of a two-point prior distribution.
\end{remark}

\begin{proof}[Proof of Proposition \ref{T:ltvol}]
We first claim that
\begin{equation*}
\E_{t, x(t,\pi)} \left[ B \Ind_{[0,\infty)}(B)\right]\to  \pi r
\end{equation*}
as $t\tend \infty$. To see this, suppose that $a > r$ and take $\gm\in(r,a)$ such that $\gamma-r <a-\gamma$. 
By Corollary~\ref{T:mon} and Proposition~\ref{T:levelcurves}, for all large enough $t$, 
\begin{IEEEeqnarray*}{rCl}
\E_{t,x(t,\pi)}\left[B\Ind_{(a,\infty)}(\Dr)\right] &\leq& 
\E_{t,\gamma t} \lt[B \Ind_{(a,\infty)}(\Dr) \rt]\\
	&=& \frac{\int_{(a, \infty)} be^{-(b-\gm)^2 \frac{t}{2}} \mu(\vd b)}
	{ \int_{\R} e^{-(b-\gm)^2 \frac{t}{2}} \mu(\vd b)} \\
	&\leq& \frac{\exp \lt(- ( a - \gm)^2 \frac{t}{2} \rt)  \int_{(a, \infty)} b \mu(\vd b)  }
	{\exp \lt( -(\gamma-r)^2 \frac{t}{2}\rt) \mu ([r, \gm))  } ,
\end{IEEEeqnarray*}
which tends to 0 as $t\to\infty$.
Now, the fact that $\P_{t,x(t,\pi)}(\Dr \in [0, r))=0$ for all $t\geq0$ finishes the claim.

Next, straightforward modifications of the arguments above show that 
\begin{equation*}
\E_{t, x(t,\pi)} \left[ B \Ind_{(-\infty,0)}(B)\right]\to (1-\pi) l
\end{equation*}
as $t\to\infty$. Since 
\begin{IEEEeqnarray*}{rCl}
\sigma(t,\pi) &=& (1-\pi)\E_{t,x(t,\pi)}\left[B \Ind_{[0,\infty)}(B)\right]-\pi \E_{t,x(t,\pi)}\left[B \Ind_{(-\infty,0)}(B)\right],
\end{IEEEeqnarray*}
this finishes the proof.
\end{proof}

\begin{remark}
Similar arguments as in the proof above show that $\mu_{t, x(t,\pi)} \Rightarrow (1-\pi)\dir{l} + \pi \dir{r}$ as $t\to\infty$.
Thus, along a level curve $x(\cdot,\pi)$ the conditional distribution of $B$ converges weakly to the two-point distribution
with mass $\pi$ at $r$ and mass $1-\pi$ at $l$.
\end{remark}
\subsection{Long-term behaviour of the boundaries}
\begin{theorem}
\item
\begin{itemize}
\item
If $l=r=0$, then $b_1(\infty)=b_2(\infty)=1/2$.
\item
If $l<r$, then $b_1(\infty)=b_1^{r-l}$ and $b_2(\infty)=b_2^{r-l}$, where $b_1^{r-l}<1/2<b_2^{r-l}$ are the optimal boundaries 
for a two-point prior distribution with mass at points separated by 0 and at a distance $r-l$ from each other.
\end{itemize}
\end{theorem}

\begin{proof}
Since the volatility $\sigma(\cdot, \cdot)$ is non-increasing in time, Proposition~\ref{T:ltvol} and Dini's theorem yield 
that $\sigma(t, \cdot)$ converges to 
$\sigma(\infty,\pi)=(r-l)\pi(1-\pi)$ uniformly on the compact interval $[b_1(0),b_2(0)]$ as $t \tend \infty$.
Therefore, given $\ep>0$ we can find $t_0$ large enough so that $\sigma(t_0,\pi)\leq (\ep+r-l)\pi(1-\pi)$ for $\pi\in[b_1(0),b_2(0)]$.
Define 
\[\hat\sigma(t,\pi):=\sigma(t,\pi) \Ind_{[b_1(0),b_2(0)]}(\pi),\]
and denote by $\hat v$ the corresponding value function. 
Since the optimal stopping problem \eqref{v} is monotone in the volatility (compare e.g. \cite[Lemma 10]{JT}), 
we have that $\hat v\geq v$. On the other hand, since $\hat\sigma=\sigma$ on the continuation region 
$\{(t,\pi):b_1(t)<\pi<b_2(t)\}$, we also have $\hat v\leq v$, so $\hat v=v$.
Moreover, by monotonicity in the volatility, 
\[v^{\ep+r-l}\leq \hat v=v\leq v^{r-l}\]
on $[t_0,\infty)\times (0,1)$,
where $v^a$ denotes the value function corresponding to a volatility function $a\pi(1-\pi)$.
Since the value function $v$ is squeezed in between the value functions $v^{\ep+r-l}$ and $v^{r-l}$
from time $t_0$, the optimal stopping boundaries $b_1$ and $b_2$ 
are squeezed in between the corresponding optimal stopping boundaries for $v^{\ep+r-l}$
and $v^{r-l}$. By inspection of the explicit formulas in the two-point distribution case, see \cite[Theorem~21.1]{PS06},
the gaps $b_1^{r-l}-b_1^{\ep+r-l}$ and $b_2^{\ep+r-l}-b_2^{r-l}$ between the boundaries vanish as $\ep\to 0$,
which finishes the proof.

\end{proof}

\begin{remark}
It is also of interest to determine $b_i(0)$ for $i=1,2$ in order to find the best bounds for the continuation region.
It seems difficult to determine these quantities in general, but an upper bound for the continuation region initially
(and thus at all times) can be established by solving the free-boundary problem for the time-homogeneous 
volatility $\sigma(0,\pi)$. However, we expect these bounds to be rather crude, and therefore 
do not provide any details.
\end{remark}

\section{The normal prior distribution}
\label{examples}

In this final section, we study the case of a normal prior distribution in further detail. In particular, we show that the
kernel in the integral equations determined in Section~\ref{S:inteqn} can be calculated explicitly
for normal priors. For similar results in the case of the two-point distribution, see \cite{GP}.

First, recall from Section~\ref{3} that a normal prior distribution with mean $m$ and variance $\gamma^2$ 
leads to a volatility surface $\sigma(\cdot, \cdot)$ 
that is symmetric around the line $\pi=1/2$. As a result, the stopping boundaries $b_1$ and $b_2$ are also 
symmetric around $\pi=1/2$ with $b_2(t) = 1-b_1(t)$, so it suffices to solve a single integral equation to determine 
both boundaries.  Next, recall that the conditional distribution $\mu_{t,x}$ is normal with standard deviation 
$\gamma(t):=\gamma/\sqrt{1+t\gamma^2}$. 
Consequently, the $x$-value that gives $\pi(t,x)=b(t)$ is such that the conditional drift equals 
\[m(t):=\Phi^{-1}(b(t))\gamma/\sqrt{1+t\gamma^2}.\]

Now, given $s>0$, let $Y$ denote a $N(m(t)s,s+s^2\gamma^2(t))$-distributed random variable. Then using (\ref{E:pinormal}), we calculate
\begin{IEEEeqnarray*}{rCl}
 K(t,s,b(t),b(t+s)) &:=& \P\left(b(t+s)<\Pi_{t+s}^{t,b(t)}<1-b(t+s)\right)\\
&=& \P\left(b(t+s) < \Phi\left(\frac{m(t) +\gamma^2(t) Y}{\gamma(t)\sqrt{1+s\gamma^2(t)}}\right) <1-b(t+s)\right)\\
&=& \Phi(d_2)-\Phi(d_1),
\end{IEEEeqnarray*}
where
\[d_1:=\frac{\Phi^{-1}(b(t+s))\gamma(t)\sqrt{1+s\gamma^2(t)}-m(t)(1+s\gamma^2(t))}
{\gamma^2(t)\sqrt{s+s^2\gamma^2(t)}}\]
and
\[d_2:=\frac{-\Phi^{-1}(b(t+s))\gamma(t)\sqrt{1+s\gamma^2(t)}-m(t)(1+s\gamma^2(t))}
{\gamma^2(t)\sqrt{s+s^2\gamma^2(t)}}.\]
Thus the kernel $K$ appearing in the integral equation \eqref{symmetriceqn} and in the corresponding 
equation for the infinite-horizon formulation is explicit. 

\begin{figure}[htbp]
\includegraphics[width=\textwidth]{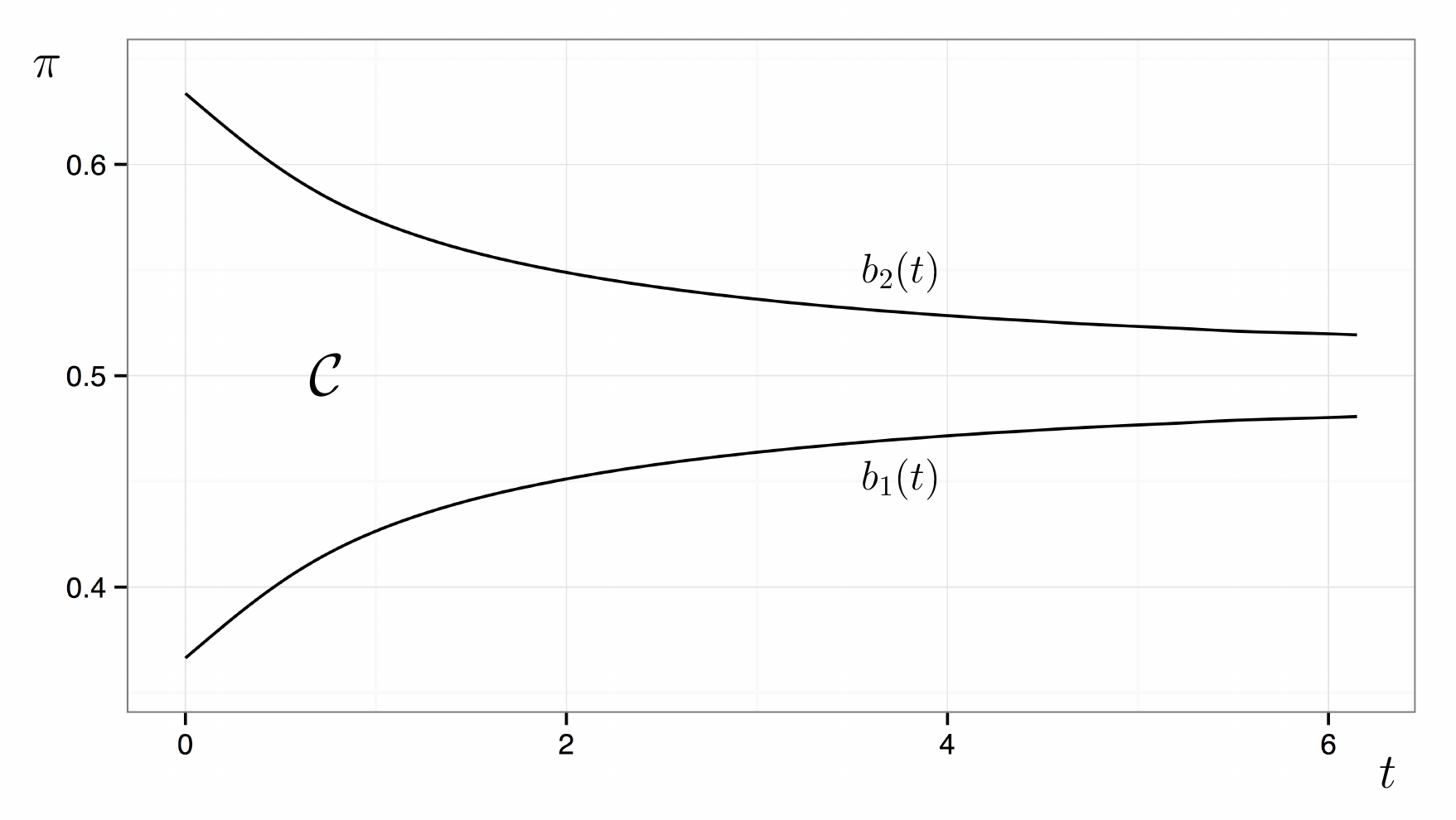}
\caption{The boundaries $b_1$ and $b_2$ calculated numerically for the case of \mbox{$N(m, 1)$-prior} (note that the boundaries do not depend on $ m \in \R$) and the cost of observation $c=0.5$ per unit time.}
\end{figure}

\end{document}